\documentclass[12pt,a4paper]{article}

\DeclareSymbolFont{AMSb}{U}{msb}{m}{n}
\DeclareSymbolFontAlphabet{\Bbb}{AMSb}

\usepackage[ansinew]{inputenc} 
\usepackage{latexsym}
\usepackage{color}

\usepackage{amstext,amsfonts,amsmath,amsthm,graphicx,amssymb,amscd,epsfig}
\usepackage{rotating}

\usepackage{subfig}
\newtheorem{theorem}{Theorem}[section]

\newtheorem{definition}[theorem]{Definition}
\newtheorem{lemma}[theorem]{Lemma}

\newcommand{\R}{\mathbb{R}}
\newcommand{\Z}{\mathbb{Z}}
\newcommand{\N}{\mathbb{N}}

\newcommand{\B}{{\cal B}}

\begin{document}

\begin{center}
{\Large Stability of quasi-simple heteroclinic cycles} \\
\mbox{} \\
{\large L.\ Garrido-da-Silva$^{2}$ and S.B.S.D.\ Castro$^{1,2,*} $}
\end{center}

\vspace{1cm}
$^*$ Corresponding author

$^{1}$ Faculdade de Economia da Universidade do Porto,
Rua Dr.\ Roberto Frias,
4200-464 Porto,
Portugal. \\
Phone: +351 225 571 100. Fax: +351 225 505 050. \\  
ORCiD: 0000-0001-9029-6893 \\
email address:  sdcastro@fep.up.pt

$^{2}$ Centro de Matem\'atica da Universidade do Porto, 
Rua do Campo Alegre 687,
 4169-007 Porto,
 Portugal. \\  
email address:  lilianagarridosilva@sapo.pt

\vspace{1cm}

\begin{abstract}
The stability of heteroclinic cycles may be obtained from the value of the local stability index along each connection of the cycle. We establish a way of calculating the local stability index for {\em quasi-simple} cycles: cycles whose connections are $1$-dimensional and contained in flow-invariant spaces of equal dimension. These heteroclinic cycles exist both in symmetric and non-symmetric contexts.
We make one assumption on the dynamics along the connections to ensure that the transition matrices have a convenient form.
Our method applies to all simple heteroclinic cycles of type $Z$ and to various heteroclinic cycles arising in population dynamics, namely non-simple heteroclinic cycles, as well as to cycles that are part of a heteroclinic network.
We illustrate our results with a non-simple cycle present in a heteroclinic network of the Rock-Scissors-Paper game.
\end{abstract}



\vspace{1.5cm}

\section{Introduction}
Heteroclinic objects (connections, cycles, and networks) exist in a robust way provided that flow-invariant spaces exist where each heteroclinic connection is of saddle-sink type. The existence of flow-invariant spaces is a natural feature of symmetric problems, in which flow-invariant spaces arise in the form of fixed-point spaces. Recent results about how symmetry can be used to construct heteroclinic cycles can be found\footnote{Our choice of bibliography aims at being illustrative rather than comprehensive.} in Ashwin and Postlethwaite \cite{AP} and in Field \cite{Field}.
Flow-invariant spaces also occur naturally in applications where the state variables are constrained by the nature of the problem. In this context, the flow-invariant spaces are not fixed-point spaces.

We address the stability of what we call {\em quasi-simple} cycles: robust heteroclinic cycles whose connections are 1-dimensional and contained in flow-invariant spaces of equal dimension. These flow-invariant spaces are not necessarily fixed-point spaces so that quasi-simple cycles may also be heteroclinic cycles arising in a non-symmetric setting.

In the context of symmetry, Krupa and Melbourne \cite{KrupaMelbourne2004} have defined a heteroclinic cycle in $\R^4$ to be simple if the connections between consecutive equilibria lie in a 2-dimensional fixed-point subspace. Simple cycles in $\R^4$ are grouped into one of three types depending on symmetry related properties: $A$, $B$ and $C$. We are not concerned with type $A$ cycles. More recently, Podvigina \cite{Podvigina2012} studies simple cycles in dimension higher than $4$ and groups the cycles into two types: $A$, as before, and $Z$. The set of type $Z$ cycles contains all cycles of types $B$ and $C$.

Heteroclinic cycles also appear in non-symmetric problems. In Evolutionary Game Theory, see Hofbauer and Sigmund \cite{HS}, the state variables are probability vectors so that the state space is an $n$-dimensional simplex. Under the classic replicator dynamics, flow-invariant spaces arise naturally on the boundary of the simplex. The classic example of a heteroclinic cycle in this context can be found in the dynamics of the Rock-Scissors-Paper (RSP) game. This game has been used to study problems in population dynamics as well as economics. See Szolnoki {\em et al.} \cite{rucklidge} and Hofbauer and Sigmund \cite{HS2003} for surveys of the former, and Hopkins and Seymour \cite{HopkinsSeymour} for the latter.
The heteroclinic cycles in the RSP game consist of 1-dimensional connections contained in flow-invariant (not fixed-point) spaces of equal dimension, see Section~\ref{sec:RSP}.

Quasi-simple heteroclinic cycles are defined to include both the symmetric and the non-symmetric cycles described above.

The study of the stability of heteroclinic objects goes back to the work of Hofbauer and Sigmund, resorting to Lyapunov functions, and of Krupa and Melbourne \cite{KrupaMelbourne95a, KrupaMelbourne95b, KrupaMelbourne2004},  and Melbourne \cite{Melbourne1991}, via the study of return maps to cross-sections to the flow.\footnote{Our study remains close to the methods of Krupa and Melbourne.} 
The stability of heteroclinic cycles and networks is obtained from the stability of individual connections and can be quantified by the stability index introduced by Podvigina and Ashwin \cite{PodviginaAshwin2011}. 
The stability index has been defined for an arbitrary invariant set. To the best of our knowledge, for heteroclinic dynamics,\footnote{See Keller \cite{Keller} and Mohd Roslan and Ashwin \cite{RoslanAshwin} for other contexts.} it has only been calculated for some types of simple cycle in \cite{PodviginaAshwin2011} and \cite{Lohse2015}, and used by \cite{CastroLohse2014} in the context of networks.
 
In the present article, we extend the calculations to quasi-simple cycles, a much wider collection of heteroclinic cycles. These include all simple cycles of type $Z$, various non-simple cycles arising from a straightforward generalisation of a construction method in \cite{AP}, and cycles with connections contained in flow-invariant spaces which are not vector subspaces, as those arising from population dynamics. 
Podvigina \cite{Podvigina2012} establishes results concerning the fragmentary asymptotic stability of cycles of type $Z$. 
Fragmentary asymptotic stability is a weak notion of stability: a fragmentarily asymptotically stable cycle attracts a positive measure set in a small neighbourhood of the cycle.
Our method\footnote{Part of our method is an adaptation of some of the techniques of Podvigina \cite{Podvigina2012}, concerning simple cycles of type $Z$, to obtain results in a far more general context.} extends to the study of essential asymptotic stability of such cycles, a stronger stability concept, as well as stability of non-simple cycles or cycles arising in a non-symmetric context. 

Our main contribution consists in providing a method that, under one mild assumption, yields an expression for the values of the local stability index at a point of a connection. Of course, in the context of simple cycles in dimension $4$, our method recovers results previously obtained by other authors. Our Assumption A guarantees basic transition matrices\footnote{The form of a basic transition matrix is given in \eqref{eq:Mj}.} exist for representing the dynamics between incoming cross-sections at consecutive equilibria and this is a crucial element in the calculation of local stability indices.
The form for the basic transition matrices under our hypotheses may be satisfied by the dynamics of some non-simple cycles. In this case, our results still hold.
We thus confirm the views of \cite[p.\ 910]{PodviginaAshwin2011} in that `transition matrices can be used to study the stability of simple cycles in higher-dimensional systems' while contradicting their expectation that `we expect such a classification to be so complex that the results can hardly be enlightening'.

Our results are constructive in the sense that we define a function whose images provide the stability index. 
We illustrate our results by calculating the local stability indices for the connections of one of the cycles in a heteroclinic network in the 4-dimensional simplex in $\R^6$. This network has been studied numerically by Sato {\em et al.} \cite{SAC}. 

Many interesting questions concerning dynamics have a natural starting point in the study of stability, to which this article contributes.
For example, stability can provide a description of the dynamics near a heteroclinic cycle in a network by making use of the stability of individual heteroclinic connections of all the cycles in the network. See Castro and Lohse \cite{CastroLohse2014}. Stability is also essential in the study of bifurcations from heteroclinic cycles, such as the studies in Postlethwaite \cite{Postlethwaite2010}, Postlethwaite and Dawes \cite{PostlethwaiteDawes}, and Lohse \cite{Lohse2015}.

In the next section, we describe some preliminary concepts and results to be used in the sequel. Section \ref{sec:stab_cycles} is divided into four subsections, including the construction of local and global maps.
The main results are the description of the transition matrices and their properties required for the calculation of the local stability indices in Subsection \ref{sec:transition_matrix}, and Theorems \ref{thm:non-negative_entries} and \ref{th:negative_entries} in the last subsection. The following section presents an example of application. 
In the appendix, we present the very long calculations required to define the essential function in the calculation of stability indices, as well as the explicit form of this function for the case of 3 variables, used in our example.

\section{Preliminaries}\label{sec:preliminaries}

Consider a dynamical system defined by 
$\dot{x}=f(x)$
where $x \in \R^n$ and $f$ is smooth. A {\em heteroclinic cycle} consists of equilibria (also called nodes), $\xi_j$, $j=1, \hdots, m$, $m\geq1$, together with trajectories which connect them:
$$
[\xi_j \rightarrow \xi_{j+1}] \subset W^u(\xi_j) \cap W^s(\xi_{j+1}) \neq \emptyset.
$$
In order to have a cycle, we assume $\xi_{m+1}=\xi_1$. 
When we have a connected union of more than one cycle, we talk about a {\em heteroclinic network}. 

Symmetry arises when there exists a compact Lie group $\Gamma$ acting on $\R^n$  such that 
$$
f(\gamma .x) = \gamma .f(x) \;\;\;\textrm{ for all } \gamma \in \Gamma, \;\;\;\textrm{ for all } x \in \R^n.
$$
We then say that $f$ is $\Gamma$-equivariant. In this case, invariant spaces for the dynamics appear naturally as {\em fixed-point spaces} of the group action: given a subgroup $\Sigma\subset\Gamma$, we define 
$$
\textrm{Fix} \left(\Sigma\right) = \{ x \in \R^n: \textrm{ }\sigma.x=x \textrm{ for all }\sigma \in \Sigma\}.
$$

Robust heteroclinic cycles in $\R^4$ have been first described as {\em simple} by Krupa and Melbourne \cite{KrupaMelbourne2004}. A cycle in $\R^4$ is robust if each connection $[\xi_j \rightarrow \xi_{j+1}]$ is of saddle-sink type and contained in a fixed-point space $P_j=\textrm{Fix}\left(\Sigma_j\right)$ for some $\Sigma_j \subset \Gamma$. 
These authors additionally define a cycle in $\R^4$ to be simple when 
\begin{enumerate}
	\item[(a)]  $\textrm{dim}\left(P_j\right)=2$ for each $j$;
	\item[(b)]  the cycle intersects each connected component of $L_j\backslash \{ 0 \}$ in at most one point, with $L_j=P_{j-1} \cap P_j$.
\end{enumerate}
However, they use implicitly the assumption that all eigenvalues of $df(\xi_j)$ are distinct. This is clarified by Podvigina \cite{Podvigina2012} who extends to higher dimensions the definition of Krupa and Melbourne \cite{KrupaMelbourne2004}. Definition 7 in \cite{Podvigina2012}, which we use in this article, states that a heteroclinic cycle in $\R^n$ is simple if for any $j$
\begin{itemize}
	\item  all eigenvalues of $df(\xi_j)$ are distinct;
	\item  $\textrm{dim}\left(P_{j-1} \ominus L_j\right)=1$.
\end{itemize}
Here $U \ominus V$ denotes the orthogonal complement to $V$ in $U$, each $P_j$ is a fixed-point space and again $L_j=P_{j-1} \cap P_j$.
When the implicit assumption that eigenvalues of $df(\xi_j)$ are distinct fails, cycles become {\em pseudo-simple} as defined by Podvigina and Chossat \cite{PC2015}. In this case, there is always at least one equilibrium so that  $df(\xi_j)$ has a double eigenvalue.

Robust simple heteroclinic cycles have further been classified into types according to some features of the isotypic decomposition. Krupa and Melbourne \cite{KrupaMelbourne2004} divided simple cycles in $\R^4$ into three classes: $A$, $B$ and $C$.
Cycles of types $B$ and $C$, see Krupa and Melbourne \cite[Definition 3.2]{KrupaMelbourne2004}, belong to type $Z$, defined by Podvigina \cite[Definition 8]{Podvigina2012}.  

We define {\em quasi-simple} cycles so as to include all simple cycles but also to be useful in a non-symmetric context. 

Let $\hat{L}_j$ be the space connecting the node $\xi_j$ to the origin in $\R^n$. In case of simple cycles, $\hat{L}_j$ coincides with $L_j=P_{j-1}\cap P_j$.

\begin{definition}\label{quasi-simple}
A {\em quasi-simple} cycle is a robust heteroclinic cycle connecting $m < \infty$ equilibria $\xi_j \in P_j\cap P_{j-1}$ so that for all $j=1, \hdots,m$
\begin{enumerate}
	\item[(i)]  $P_j$ is a flow-invariant space,
	\item[(ii)]  $\mbox{dim }P_j=\mbox{dim }P_{j+1}$,
	\item[(iii)]  $\mbox{dim } (P_j \ominus \hat{L}_j)=1$.
\end{enumerate}
\end{definition}
Condition (iii) ensures that $[\xi_j \rightarrow \xi_{j+1}] \subset P_j$ is 1-dimensional. 

We stress that the flow-invariant spaces in Definition~\ref{quasi-simple} do not have to be fixed-point spaces. In fact, they do not have to be invariant by symmetry.
In particular, many heteroclinic cycles arising from replicator dynamics are quasi-simple. In this case, the flow-invariance property emerges from  restrictions on the state space, not because there is any symmetry present.
In Section \ref{sec:RSP}, we address an example where these flow-invariant spaces are 2-dimensional but neither are they vector spaces nor fixed-point spaces. 
Note also that, based on Podvigina \cite[Lemma 1]{Podvigina2013}, all simple heteroclinic cycles are quasi-simple.

In order to describe the stability of a cycle we calculate the stability index at a point on each of its connections. The stability index was defined by Podvigina and Ashwin \cite{PodviginaAshwin2011} to characterize the local geometry of basins of attraction of heteroclinic cycles. The stability index is constant along trajectories and can be computed with respect to a suitable transverse section to the flow, as shown in Theorems 2.2 and 2.4 in \cite{PodviginaAshwin2011}, respectively.
The stability index and its local version are defined in Definition 5 and section 2.3 in Podvigina and Ashwin \cite{PodviginaAshwin2011}, which we rewrite in Definition \ref{def:stab_index} below. 
As usual, we denote by $B_{\varepsilon}(X)$ an $\varepsilon$-neighbourhood of a (compact, invariant) set $X \subset \R^n$. We write $\B(X)$ for the basin of attraction of $X$, i.e. the set of points $x \in \R^n$ with $\omega(x) \subset X$. For $\delta>0$ the $\delta$-local basin of attraction of $X$ is $\B_\delta(X):=\{x \in \B(X):\textrm{ } \phi_t(x) \in B_\delta(X) \textrm{ for all } t>0  \}$, where $\phi_t(.)$ is the flow generated by the system of equations. By $\ell(.)$ we denote Lebesgue measure.

\begin{definition}[\cite{PodviginaAshwin2011}]\label{def:stab_index}
For $x \in X$ and $\varepsilon, \delta >0$ define 
$$
\Sigma_\varepsilon(x):=\frac{\ell(B_\varepsilon(x) \cap \B(X))}{\ell(B_\varepsilon(x))}, \qquad \Sigma_{\varepsilon,\delta}(x):=\frac{\ell(B_\varepsilon(x) \cap \B_\delta(X))}{\ell(B_\varepsilon(x))}.
$$
Then the {\em stability index} at $x$ with respect to $X$ is defined to be
$$
\sigma(x):=\sigma_+(x)-\sigma_-(x),
$$
where
$$
\sigma_-(x):= \lim\limits_{\varepsilon \to 0} \left[ \frac{\textnormal{ln}(\Sigma_\varepsilon(x) )}{\textnormal{ln}(\varepsilon)}  \right], \qquad \sigma_+(x):= \lim\limits_{\varepsilon \to 0} \left[ \frac{\textnormal{ln}(1-\Sigma_\varepsilon(x) )}{\textnormal{ln}(\varepsilon)}  \right].
$$
The convention that $\sigma_-(x)=\infty$ if $\Sigma_\varepsilon(x)=0$ for some $\varepsilon>0$ and $\sigma_+(x)=\infty$ if $\Sigma_\varepsilon(x)=1$ is introduced. Therefore, $\sigma(x) \in [-\infty, \infty]$. In the same way the {\em local stability index} at $x \in X$ is defined to be
$$
\sigma_{\textnormal{loc}}(x):=\sigma_{\textnormal{loc},+}(x)-\sigma_{\textnormal{loc},-}(x),
$$
with
$$
\sigma_{\textnormal{loc},-}(x):= \lim\limits_{\delta \to 0} \lim\limits_{\varepsilon \to 0} \left[ \frac{\textnormal{ln}(\Sigma_{\varepsilon,\delta}(x))}{\textnormal{ln}(\varepsilon)}  \right], \: \sigma_{\textnormal{loc},+}(x):= \lim\limits_{\delta \to 0} \lim\limits_{\varepsilon \to 0} \left[ \frac{\textnormal{ln}(1-\Sigma_{\varepsilon,\delta}(x))}{\textnormal{ln}(\varepsilon)}  \right].
$$
\end{definition}

We say that the stability index $\sigma_{\textnormal{loc}}(x)$ quantifies the local extent of the local basin of attraction of $X$ in the following sense. For $\sigma_{\textnormal{loc}}(x)>0$, an increasingly large portion of points in an $\varepsilon$-neighbourhood of $x$ converge to $X$ as $\varepsilon \rightarrow 0$. For $\sigma_{\textnormal{loc}}(x)<0$, the measure of the set of such points goes to zero.

Conditions for asymptotic stability of heteroclinic cycles are given by Krupa and Melbourne \cite{KrupaMelbourne95a, KrupaMelbourne2004}. 
When a heteroclinic cycle is part of a network, asymptotic stability is impossible. This in turn requires intermediate notions of stability. We are interested in two forms of non-asymptotic stability: {\em essential asymptotic stability} and {\em fragmentary asymptotic stability}, of which the first is the strongest.

\begin{definition}[Definition 1.2 in \cite{Brannath}] \label{def_eas}
A compact invariant set $X$ is called {\em essentially asymptotically stable (e.a.s.)} if it is asymptotically stable relative to a set $N \subset \R^n$ with the property that
$$
\lim\limits_{\varepsilon \to 0} \frac{\ell(B_{\varepsilon}(X) \cap N)}{\ell(B_\varepsilon(X))} = 1.
$$
\end{definition}
\begin{definition}[Definition 2 in \cite{Podvigina2012}] \label{def_fas}
A compact invariant set $X$ is called {\em fragmentarily asymptotically stable (f.a.s.)} if for any $\delta >0$
$$
\ell(\B_{\delta}(X)) >0.
$$
\end{definition}

The local stability index relates to the stability of a compact invariant set $X$ (a cycle or network) thanks to the following two results where the former is the local version of Theorem 2.3 in \cite{PodviginaAshwin2011}. See also Lohse \cite[proof of Theorem 3.1]{Lohse2015a} for a relation between local and global basins of attraction. 

\begin{lemma}\label{index_fas}
Suppose that for $x \in X$ the local stability index is defined and such that $-\infty < \sigma_{\textnormal{loc}}(x)$. Then X is f.a.s.
\end{lemma}

\begin{proof}
If $-\infty < \sigma_{\textnormal{loc}}(x)$ then by definition its local basin of attraction contains a set of positive measure. This is just the definition of f.a.s.
\end{proof}

\begin{theorem}[Theorem 3.1 in \cite{Lohse2015}]\label{thm:lohse}
Let $X \subset \R^n$ be a heteroclinic cycle or network with finitely many equilibria and connecting trajectories. Suppose that\footnote{Here $\ell_1(.)$ denotes 1-dimensional Lebesgue measure.} $\ell_1(X) < \infty$ and that the local stability index $\sigma_{loc}(x)$ exists and is not equal to zero for all $x \in X$. Then, generically, we have $X$ is e.a.s. is equivalent to $\sigma_{loc}(x)>0$ along all connecting trajectories.
\end{theorem}

In what follows we drop the index ``loc'' for clarity, as we always refer to the local stability index.

For a (robust) quasi-simple cycle, recall that $\hat{L}_j$ is the space connecting the node $\xi_j$ to the origin in $\R^n$. When the state space is not Euclidean space, $\R^n$ is the Euclidean space of smallest dimension containing the state space. See Section \ref{sec:RSP} for such an example.
We then group the eigenvalues into four types:
\begin{itemize}
	\item  {\em radial} eigenvalues $(-r<0)$, which have eigenvectors in $\hat{L}_j$;
	\item  {\em contracting} eigenvalues $(-c<0)$, which have eigenvectors in $P_{j-1} \ominus \hat{L}_j$;
	\item  {\em expanding} eigenvalues $(e>0)$, which have eigenvectors in $P_{j} \ominus \hat{L}_j$;
	\item  {\em transverse} eigenvalues $(t\in\R)$, otherwise.
\end{itemize}
Conditions (ii) and (iii) in Definition~\ref{quasi-simple} guarantee that contracting and expanding eigenvalues are simple.

It may be that no radial and/or transverse eigenvalues exist. The two-player Rock-Scissors-Paper game in the last section provides an example where no radial eigenvalues exist because $\hat{L}_j$ is not in the state space. A cycle in the one-player Rock-Scissors-Paper in the 2-dimensional simplex has neither radial nor transverse eigenvalues.

We focus on vector fields whose eigenvalues of the Jacobian matrix at the equilibria are all real. However, considering Footnote 1 in \cite{Podvigina2012}, which also follows from the form of the local maps below, having complex radial eigenvalues does not change stability results.

\section{Stability of quasi-simple cycles}\label{sec:stab_cycles}

The standard way to address stability issues of heteroclinic cycles is to look at return maps to cross-sections placed along heteroclinic connections around the cycle. These are obtained by composition of local and global maps. The global maps take points from a neighbourhood of one equilibrium to a neighbourhood of the following equilibrium along the cycle. The local maps describe the trajectories of points near an equilibrium and depend only on the eigenvalues at that equilibrium. Given Lemma \ref{index_fas} and Theorem \ref{thm:lohse}, we can obtain information about the stability of a heteroclinic cycle by looking at the local stability indices along its connections.

\subsection{Maps between cross-sections}

Define cross-sections to the flow near an equilibrium $\xi_j$ and denote them by $H_j^{\tiny{in}}$, along an incoming connection, and by $H_j^{\tiny{out}}$, along an outgoing connection. If the cycle is part of a network and the equilibrium has more than one incoming/outgoing connection, we distinguish them with a second index. Thus, $H_j^{\tiny{in},i}$ denotes the cross-section, near $\xi_j$, to the connection $[\xi_i \to \xi_j]$ whereas $H_j^{\tiny{out},i}$ denotes the cross-section, near $\xi_j$, to the connection $[\xi_j \to \xi_i]$. \emph{Local maps} $\phi_j$, near $\xi_j$, and \emph{global maps} $\psi_j$, along a connection $[\xi_j \to \xi_{j+1}]$, are such that
$$
\phi_j: \; H_j^{\tiny{in}} \rightarrow H_j^{\tiny{out}} \;\;\; \textrm{ and } \;\;\; \psi_j: \; H_j^{\tiny{out}} \rightarrow H_{j+1}^{\tiny{in}} ,
$$
respectively.
It is well-known that not all dimensions are relevant in the study of stability and these maps can be analysed in dimension lower than that of the state space (see \cite{KrupaMelbourne2004}). For $1$-dimensional connections the relevant dimension is equal to the number of transverse eigenvalues plus one. For a quasi-simple cycle the number of transverse eigenvalues is the same at every equilibrium so that the composition of these maps is well-defined. 

Let $n_t$ be the number of transverse eigenvalues at each node. We can then restrict the cross-sections above to an $(n_t+1)$-dimensional subspace. In order to keep notation sufficiently simple, we preserve the use of $H_j^{\tiny{in}}$ and $H_j^{\tiny{out}}$ for the $(n_t+1)$-dimensional subspace of each of them. From now on, we always work in the relevant $(n_t+1)$-dimensional subspace. We further write $\R^{1+n_t}$ or $\R^N$ ($N=n_t+1$) in place of a cross-section when no restrictions are applied and no confusion arises. 
In this subspace, we use the coordinates $(w,\boldsymbol{z})$, where $w$ is related to the expanding direction at $\xi_j$, $\boldsymbol{z}$ is $n_t$-dimensional and related to the transverse directions to the connection $[\xi_j \rightarrow \xi_{j+1}]$.

The composition $g_j=\psi_j \circ \phi_j$ describes the trajectory of points from $H_j^{\tiny{in}}$ to $H_{j+1}^{\tiny{in}}$. 
The local maps $\phi_j$ are obtained by looking at the linearisation of the flow near $\xi_j$, as presented in the next subsection. The global maps $\psi_j$ are homeomorphisms and depend on the type of connection along which they describe the dynamics. 

In what follows, we make the following assumption:

\paragraph{Assumption A:} The global maps consist of a rescaled permutation of the local coordinate axes.
\bigbreak

The results in \cite{KrupaMelbourne2004} and \cite{Podvigina2012} are a consequence of the symmetry of the problem and provide a way of checking Assumption A. In fact, Assumption A always holds for simple cycles of type $Z$. In the context of non-simple cycles, \cite[Remark, p.\ 1603]{AC} also argue that the global maps are a permutation, without rescaling.

A \emph{return} or \emph{Poincar\'e map} for a cycle connecting $m$ equilibria is given by
\begin{equation}\label{pi_map}
\pi_j:H^{in}_j \rightarrow H^{in}_j, \;\;\;\; \pi_j=g_{j-1} \circ \hdots \circ g_1 \circ g_m \circ \hdots \circ g_{j+1} \circ g_j.
\end{equation}
Stability properties of an invariant object can be obtained by looking at its $\delta$-local basin of attraction. 
In order to establish any type of stability of a cycle, we make use of the fact that the $\delta$-local basin of attraction of a cycle can be related to that of a fixed-point (without loss of generality, the origin) of a suitable collection of return maps. The study of stability can be achieved by iterating the return maps as well as the maps we call `partial turns', described as follows 
\[
g_{\left(l,j\right)}: H_j^{in} \rightarrow H_{l+1}^{in}, \;\;\;\;
g_{\left(l,j\right)}=\begin{cases}
g_{l}\circ\cdots\circ g_{j}, & l>j\\
g_{l}\circ\cdots\circ g_{1}\circ g_{m}\circ\cdots\circ g_{j}, & l<j\\
g_{j}, & l=j.
\end{cases}
\]
Using coordinates $(w,\boldsymbol{z}) \in \R^{1+n_t}$ in $H_j^{in}$, for $\delta>0$ we define the $\delta$-local basin of attraction of the origin in $\R^{1+n_t}$ for the map $\pi_j$ (see \cite[Definition 10]{Podvigina2012}) to be
\begin{equation}
\begin{aligned}\mathcal{B}_{\delta}^{\pi_{j}}=\Big\{&  \left(w,\boldsymbol{z}\right)\in\mathbb{R}^{1+n_t}:\left\Vert g_{\left(l,j\right)}\circ\pi_{j}^{k}\left(w,\boldsymbol{z}\right)\right\Vert <\delta\textrm{ for all } l=1,\ldots,m,\;k\in\mathbb{N}_{0}  \\
&  \left. \textrm{ and } \lim_{k\rightarrow \infty} \left\Vert g_{\left(l,j\right)}\circ\pi_{j}^{k}\left(w,\boldsymbol{z}\right) \right\Vert =0\textrm{ for all } l=1,\ldots,m \right\}.
\end{aligned}
\label{eq:Basin}
\end{equation}
Note that $g_{\left(l,j\right)}\circ\pi_{j}^{k}:H_j^{\tiny{in}} \rightarrow H_{l+1}^{\tiny{in}}$ contains $k$ full returns from $\xi_j$ plus a partial turn up to $\xi_{l+1}$.

\subsection{The local maps}

The local maps $\phi_j: H_j^{\tiny{in}} \rightarrow H_j^{\tiny{out}}$ are constructed by integrating the flow linearised about $\xi_j$. 
Because all eigenvalues are assumed to be real, for a quasi-simple cycle there are one contracting eigenvalue ($-c_j$), one expanding eigenvalue ($e_j$), $n_r$ radial eigenvalues ($-r_{j,l}$, $l=1, \hdots, n_r$), and $n_t=n-n_r-2$ transverse eigenvalues ($t_{j,s}$, $s=1, \hdots, n_t$). 
The constants $c_j$, $e_j$ and $r_{j,l}$ are positive but $t_{j,s}$ can have either sign. We assume that Ruelle's \cite{ruelle} sufficient condition for linearization of the flow around each node is satisfied.
Restricting to the relevant $(n_t+1)$-dimensional subspace of the cross-sections, 
we obtain\footnote{Because all points in each orthant of $\R^{n-1}$ of $H^{in}_j$ follow the same path (see \cite[p. 1894]{Podvigina2012}), we only consider the dynamics in the positive orthant.}
\begin{equation}\label{eq:local_map}
\phi_j(w,\boldsymbol{z}) = \left( v_{0,j}w^{\frac{c_j}{e_j}}, z_1w^{-\frac{t_{j,1}}{e_j}}, \hdots, z_{n_t}w^{-\frac{t_{j,n_t}}{e_j}} \right)
\end{equation}
where $v_{0,j}$ accounts for the coordinate of the initial condition in the contracting direction at $\xi_j$. If $t_{j,s}$ is positive for some $s=1, \hdots, n_t$, then the domain of $\phi_j$ needs to be restricted to a subset of $H_j^{\tiny{in}}$.

\subsection{Transition matrices}\label{sec:transition_matrix}

The maps $g_j:H^{in}_j \rightarrow H^{in}_{j+1}$ may conveniently be expressed in terms of $N \times N$ matrices ($N=n_t+1$, as usual) as follows
\begin{equation}
g_{j}\left(w,\boldsymbol{z}\right)=A_{j}B_j\left[
\begin{array}{c}
v_{0,j}w^{{\frac{c_{j}}{e_{j}}}}\\
\left\{ z_{s}w^{-\frac{t_{j,s}}{e_{j}}}\right\} 
\end{array}\right], 
\;\;s=1,\ldots,n_{t},\;\;j=1,\ldots,m,
\label{eq:gj}
\end{equation}
where the column-vector gives the image of $(w,\boldsymbol{z})$ by $\phi_j$ in \eqref{eq:local_map}.
Following Assumption A, the matrices $A_j$ are permutation matrices and the matrices $B_j$ are diagonal and represent a rescaling of coordinates:
$$
B_j=\left[\begin{array}{cccc}
a_{j,1} & 0 &  \ldots & 0\\
0 & a_{j,2} &  \ldots & 0\\
. & . &  \ldots & .\\
0 & 0 &  \ldots & a_{j,N}
\end{array}\right].
$$
In particular, all constants $a_{j,i}$, $i=1,\ldots,N$, are positive due to invariance of $P_j$ and uniqueness of solutions. 

Recall that when the global map in Assumption A is the identity, the matrix $A_j$ need not be the identity since it also accounts for the permutation between the local bases in $H_j^{\tiny{out}}$ and $H_{j+1}^{\tiny{in}}$.  For instance, for simple cycles of type $B$ the matrix $A_j$ is the identity; however, for simple cycles of type $C$, it is not. When the cycle is simple and $n_t=1$, the global maps are given in \cite[Proposition 4.1]{KrupaMelbourne2004}.

As in \cite{Podvigina2012}, we use new coordinates of the form\footnote{Below we always use boldtype to indicate an $N$-dimensional vector. In particular, $\boldsymbol{-\infty}= \left(-\infty, \hdots, -\infty\right)$.}
\begin{equation}\label{eq:change}
\boldsymbol{\eta}=\left(\ln w,\ln z_{1},\ldots,\ln z_{n_{t}}\right).
\end{equation}
The maps $g_j:\boldsymbol{\eta}\mapsto\boldsymbol{\eta}$ become linear and we denote them by $\mathcal{M}_{j}$ where
\begin{equation}\label{eq:poincare}
\mathcal{M}_{j}\boldsymbol{\eta}=M_{j}\boldsymbol{\eta}+F_{j}
\end{equation}
with $M_j$ and $F_j$ given by
\begin{equation}\label{eq:Mj}
M_{j}=A_{j}\left[\begin{array}{ccccc}
b_{j,1} & 0 & 0 & \ldots & 0\\
b_{j,2} & 1 & 0 & \ldots & 0\\
b_{j,3} & 0 & 1 & \ldots & 0\\
. & . & . & \ldots & .\\
b_{j,N} & 0 & 0 & \ldots & 1
\end{array}\right],\;\;\; F_{j}=A_j \left[\begin{array}{c}
\ln v_{0,j}+\ln a_{j,1}\\
\ln a_{j,2}\\
\vdots\\
\ln a_{j,N}
\end{array}\right].
\end{equation}
Due to Assumption A, the matrix $A_j$ is a permutation matrix and remains the same in both \eqref{eq:gj} and \eqref{eq:Mj}.

The entries of $M_j$ depend on the eigenvalues at node $\xi_j$ as follows
$$
b_{j,1}=\frac{c_{j}}{e_{j}}, \;\; b_{j,s+1}=-\frac{t_{j,s}}{e_{j}}, \;\; s=1,\ldots,n_{t}, \;\; j=1,\ldots,m.
$$

In \cite{Podvigina2012}, the matrices $M_j$ are called {\em basic transition matrices} of the maps $g_j$. Given the form of the local maps in \eqref{eq:local_map} and the change of coordinates \eqref{eq:change} we have just proved the following 

\begin{theorem}\label{th:matrix}
For a quasi-simple heteroclinic cycle satisfying Assumption A, the transition between incoming sections at consecutive equilibria can be described using basic transition matrices of the form \eqref{eq:Mj}.
\end{theorem}

This representation in matrix form is essential for our results. In fact, if a cycle is either non-quasi-simple or does not satisfy Assumption A but  such a representation exists, then our results still hold.
 
We note that the matrices $M_j$ depend only on the contracting, expanding and transverse eigenvalues at the corresponding node $\xi_j$.
The dimension of a transition matrix, of course, depends on the dimension of the state space. The form of its entries however does not.

In the coordinates defined in \eqref{eq:change}, the transition matrices of the maps $\pi_{j}$
and $g_{\left(j,l\right)}$ are the following products of basic transition matrices
\begin{equation}\label{eq:transitionM}
M^{\left(j\right)}=M_{j-1}\ldots M_{1}M_{m}\ldots M_{j+1}M_{j}
\end{equation}
and 
$$
M_{\left(l,j\right)}=\begin{cases}
M_{l}\ldots M_{j}, & l>j\\
M_{l}\ldots M_{1}M_{m}\ldots M_{j}, & l<j\\
M_{j}, & l=j,
\end{cases}
$$
respectively. 

In the definition of local stability indices, we assume asymptotically small $w$
and $z_{s}$ for any $s=1,\ldots,n_{t}$, which is equivalent to
asymptotically large negative $\boldsymbol{\eta}$. As observed in \cite[p. 903]{PodviginaAshwin2011}, taking into account that $F_j$ in \eqref{eq:poincare} is finite, we can ignore it\footnote{Alternatively, we may proceed as in \cite{Postlethwaite2010} and rescale the coordinates such that all the coefficients of the map $g_j$ are one, which yields $F_j=\boldsymbol{0}$.} so that the map $g_j$ asymptotically to leading order is described by the matrix $M_j$. The same holds true for their compositions. Then the limit in \eqref{eq:Basin} in the new coordinates \eqref{eq:change} becomes
\begin{equation}\label{eq:basinlog}
M_{\left(l,j\right)}\left(M^{\left(j\right)}\right)^{k}\boldsymbol{\eta}\underset{{\scriptstyle k\rightarrow\infty}}{\rightarrow}\boldsymbol{-\infty} \;\;\;\textrm{ for all } l=1,\ldots,m. 
\end{equation}
Setting $M=M^{\left(j\right)}$, the points in $\mathcal{B}^{\pi_j}_{\delta}$ are a subset of
\begin{equation}\label{eq:basinlog2}
U^{-\infty}\left(M\right)=\left\{ \boldsymbol{y}\in\mathbb{R}_{-}^{N}:\;\lim_{k\rightarrow\infty}M^{k}\boldsymbol{y}=\boldsymbol{-\infty}\right\}.
\end{equation}
Assume that $M$ has eigenvalues $\lambda_{1},\ldots,\lambda_{N}$ and corresponding linearly independent eigenvectors $\boldsymbol{w}_{1},\ldots,\boldsymbol{w}_{N}$.
Let $\lambda_{\max}$ be the maximum, in absolute value, eigenvalue
of $M$ such that $\left|\lambda_{\max}\right|\neq1$ and $\boldsymbol{w}^{\max}=\left(w_{1}^{\max},\ldots,w_{N}^{\max}\right)$ the associated eigenvector. 

Although \cite{Podvigina2012} claims that 
Lemma 3 provides only sufficient conditions in terms of $\lambda_{\max}$ and $\boldsymbol{w}^{\max}$ that guarantee $\ell\left(U^{-\infty}\left(M\right)\right)>0$ for an arbitrary $N \times N$ real matrix $M$ (regardless of whether or not $\left|\lambda_{\max}\right|\neq1$), it is easy to see that these are also necessary. For ease of reference, we state the necessary and sufficient conditions in  Lemma~\ref{lemma5}.

Note that any vector $\boldsymbol{y}=\left(y_1,\ldots,y_N\right)\in\R_{-}^{N}$ may be written as
a linear combination of the eigenvectors of $M$ as
\begin{equation}\label{eq:lincomb}
\boldsymbol{y}=\sum_{i=1}^{N}a_{i}\boldsymbol{w}_{i}=P\boldsymbol{a},
\end{equation}
where $P$ is the matrix whose columns are $\boldsymbol{w}_{i}$, $i=1,\ldots,N$, and $\boldsymbol{a}=\left(a_{1},\ldots,a_{N}\right)\in\R^{N}$ is some vector.
Thus, for each $k\in\N$, the $k^{th}$ iterate of $\boldsymbol{y}$ under $M$ is
\begin{equation}\label{iterates}
M^k\boldsymbol{y}=\sum_{i=1}^{N}\lambda_i^k a_{i}\boldsymbol{w}_{i}.
\end{equation}
\begin{lemma}[Lemma 3 in \cite{Podvigina2012}]\label{lemma5}
The measure $\ell\left( U^{-\infty}(M)\right)$ is positive if and only if the three following conditions are satisfied:
\begin{enumerate}
\item [{(i)}]$\lambda_{\max}$ is real;
\item [{(ii)}]$\lambda_{\max}>1$;
\item [{(iii)}]$w_{l}^{\max}w_{q}^{\max}>0$ for all $l,q=1,\ldots,N$.
\end{enumerate}
\end{lemma}
\begin{proof} 
Suppose at least one of the conditions \emph{(i)-(iii)} is violated. It follows directly from Lemma 3(i)-(iv) in \cite{Podvigina2012} that $\ell\left(U^{-\infty}\left(M\right)\right)=0$.  

The converse is already shown in Lemma 3(v) of \cite{Podvigina2012}.
\end{proof}

Since the transformation \eqref{eq:change} can turn a set of finite measure into one of infinite measure, we preserve the convention of \cite[p.\ 1900]{Podvigina2012} that the measure of a set is always its measure in the original variables.

\subsection{Calculation of stability indices}

For $R \in \R$, consider the subset of $\R^N$ 
\begin{equation}\label{eq:UR}
U_{R}=\left\{ \boldsymbol{y}:\;\max_{i}\textrm{ }y_{i}<R\right\}.
\end{equation}
Using \eqref{eq:basinlog} and recalling that the change of coordinates \eqref{eq:change} transforms the origin into $\boldsymbol{-\infty}$, an alternative way of describing $\mathcal{B}_{\delta}^{\pi_{j}}$ for $S=\ln(\delta)$ is
\begin{equation}\label{eq:S-localbasin}
\begin{aligned}
\mathcal{U}_{S}^{M^{\left(j\right)}}=& \left\{ \boldsymbol{y}\in\mathbb{R}_{-}^{N}:\;M_{\left(l,j\right)}\left(M^{\left(j\right)}\right)^{k}\boldsymbol{y}\in U_{S}\textrm{ for all } l=1,\ldots,m,\;k\in\mathbb{N}_{0}\right. \\
& \textrm{ and }
\left. \lim_{k\rightarrow\infty}M_{\left(l,j\right)}\left(M^{\left(j\right)}\right)^{k}\boldsymbol{y}=\boldsymbol{-\infty} \textrm{ for all } l=1,\ldots,m \right\} .
\end{aligned}
\end{equation}
The set $\mathcal{U}_{S}^{M^{\left(j\right)}}$ is the
$S$-local basin of attraction of $\boldsymbol{-\infty}$
in $\mathbb{R}^{N}$ for the matrix $M^{\left(j\right)}$ representing the map $\pi_j$.

The calculation of the local stability indices is split into two cases: when transverse eigenvalues at all nodes are negative and when, for at least one node, at least one transverse eigenvalue is positive. This latter case occurs necessarily when the cycle is part of a heteroclinic network.
In terms of the transition matrices, this corresponds to basic transition matrices with only non-negative entries and with at least one negative entry, respectively. These two cases are addressed in Theorems~\ref{thm:non-negative_entries} and \ref{th:negative_entries} below.

The following is a useful auxiliary result for the proof of Theorem~\ref{thm:non-negative_entries}.

\begin{lemma}\label{lem:lambda_greater_1}
Let $M$ be a transition matrix \eqref{eq:transitionM} with non-negative entries and $\lambda_{\max}>1$. Then generically $U^{-\infty}(M) = \R^N_-$.
\end{lemma}

\begin{proof}
Since $\left|\lambda_{\max}\right|>1$, Theorem 3 in
\cite{Podvigina2012} guarantees that generically all components of the eigenvector $\boldsymbol{w}^{\max}$ are non-zero, and by Lemma 4(iii) in \cite{Podvigina2012},
$U^{-\infty}\left(M\right)=\mathbb{R}_{-}^{N}$.
\end{proof}

We point out that, in view of Lemma 4(i) in \cite{Podvigina2012}, the hypothesis that $\lambda_{\max}>1$ is equivalent to $\left|\lambda_{\max}\right|>1$. 

Denote by $\sigma_j$ the local stability index along the connection leading to $\xi_j$.

\begin{theorem}\label{thm:non-negative_entries}
Let $M_{j}$, $j=1,\ldots,m$, be basic transition matrices of a collection of maps
associated with a heteroclinic cycle.
Suppose that for all $j=1,\ldots,,m$ all entries of the matrices
are non-negative. Then:
\begin{enumerate}
\item [(a)]  If the transition matrix $M^{\left(1\right)}=M_{m}\ldots M_{1}$
satisfies $\left|\lambda_{\max}\right|>1$ then $\sigma_{j}=+\infty$
for all $j=1,\ldots,m$, and the cycle
is asymptotically stable.
\item [(b)]  Otherwise, $\sigma_{j}=-\infty$ for
all $j=1,\ldots,m$ and the cycle is not an attractor. 
\end{enumerate}
\end{theorem}

\begin{proof}
We treat each case separately.
\begin{enumerate}
\item[(a)]Suppose that the matrix $M\equiv M^{\left(1\right)}$
satisfies $\left|\lambda_{\max}\right|>1$. Similarity of the matrices $M^{(j)}$ ensures that the same holds for all $j=1, \hdots ,m$.

The product of matrices with non-negative entries also has non-negative entries. This is the case for $M^{(j)}$ and $M_{(l,j)}$ with $j,l=1,\ldots,m$. By virtue of Lemma~\ref{lem:lambda_greater_1}, $U^{-\infty}(M^{(j)}) = \R^N_-$. Then for any $\boldsymbol{y}\in \R^N_{-}$ 
$$
M_{\left(l,j\right)}\left(M^{\left(j\right)}\right)^{k}\boldsymbol{y}\in\R^N_{-} \;\;\; \textrm{ for all }j,l=1,\ldots,m,\; k\in\N_0
$$ 
and
$$
\lim_{k\rightarrow\infty}M_{\left(l,j\right)}\left(M^{\left(j\right)}\right)^{k}\boldsymbol{y}=\boldsymbol{-\infty} \;\;\; \textrm{ for all } j,l=1,\ldots,m.
$$

Taking $U_R$ an $R$-neigbourhood of $\boldsymbol{-\infty}$ with $R<0$, it follows that $U_{R}\cap\mathcal{U}_{S}^{M^{\left(j\right)}}$ is reduced to $U_{R}$ for sufficiently large negative $R<S$. Taking the respective measures in original coordinates,
$\sigma_{j,+}=\infty$
and $\sigma_{j,-}=0$ for all $j=1,\ldots,m$, by the convention in Definition~\ref{def:stab_index}.

\item[(b)]If the matrix $M\equiv M^{\left(1\right)}$ satisfies $\left|\lambda_{\max}\right|\leq1$,
then $U^{-\infty}\left(M\right)$ is empty as proved in \cite[Lemma 3(i)]{Podvigina2012}. Hence, $\sigma_{1,+}=0$ and $\sigma_{1,-}=\infty$. Since the inequality $\left|\lambda_{\max}\right|\leq1$ is satisfied for
any $M^{\left(j\right)}$, we have $\sigma_{j,+}=0$ and $\sigma_{j,-}=\infty$
for all $j=1,\ldots,m$.

\end{enumerate}
\end{proof}

Irrespectively of the sign of the entries of the basic transition matrices $M_j$ we have the following generalization of Corollary 4.1 in \cite{PodviginaAshwin2011} to quasi-simple heteroclinic cycles.

\begin{lemma}\label{lem:l(U(M))}
Let $M_{j}$, $j=1,\ldots,m$, be  basic transition matrices of a collection of maps associated with a heteroclinic cycle.
If $\ell\left(U^{-\infty}\left(M^{\left(j\right)}\right)\right)=0$ for some $j\in\left\{ 1,\ldots,m\right\}$ then $\ell\left(\mathcal{U}_{S}^{M^{\left(j\right)}}\right)=0$ for all $j=1,\ldots,m$ and $S<0$.
\end{lemma}

\begin{proof}
By virtue of Lemma \ref{lemma5}, $\ell\left(U^{-\infty}\left(M^{\left(j\right)}\right)\right)=0$ if and only if $M^{\left(j\right)}$ violates at least one condition from \emph{(i)} to \emph{(iii)}. 

We note that all matrices $M^{\left(j\right)}$ are similar so that conditions \emph{(i)-(ii)} of Lemma \ref{lemma5} either are, or are not, simultaneously satisfied for all $j=1,\ldots,m$. If either \emph{(i)} or \emph{(ii)} fail, we trivially get 
$\ell\left(U^{-\infty}\left(M^{\left(j\right)}\right)\right)=0$ for all $j=1,\ldots,m$. Since $\mathcal{U}_{S}^{M^{\left(j\right)}}\subset U^{-\infty}\left(M^{\left(j\right)}\right)$, it follows that $\ell\left(\mathcal{U}_{S}^{M^{\left(j\right)}}\right)=0$ for all $j=1,\ldots,m$ and $S<0$.

Suppose now \emph{(i)-(ii)} hold for all $j=1,\hdots,m$ while \emph{(iii)} is not satisfied for some $j\in\left\{ 1,\ldots,m\right\}$.
That is, $\lambda_{\max}$ is real and greater than 1 and there exist $q,p\in\left\{ 1,\ldots,N\right\}$ such that $w^{\max,j}_{q}w^{\max,j}_{p}\leq0$. 
Considering expansion \eqref{eq:lincomb} for a vector $\boldsymbol{y}$, the iterates $\left(M^{(j)}\right)^k\boldsymbol{y}$ become asymptotically close to $a_{\max}\lambda_{\max}^k\boldsymbol{w}_{\max}$ as $k\rightarrow\infty$ (see \eqref{iterates}t). Then $\left(M^{(j)}\right)^k\boldsymbol{y}$ are not in $\R_{-}^N$ for sufficiently large $k$ and any $\boldsymbol{y}\in\R^N$.
In particular, this applies to the iterates of $M_{\left(j-1,l\right)}\boldsymbol{y}$ under $M^{\left(j\right)}$ such that 
$$
\lim_{k\rightarrow\infty}\left(M^{\left(j\right)}\right)^{k}M_{\left(j-1,l\right)}\boldsymbol{y}=\lim_{k\rightarrow\infty}M_{\left(j-1,l\right)}\left(M^{\left(l+1\right)}\right)^{k}\boldsymbol{y}\notin\mathbb{R}_{-}^{N}.
$$
for all $l=1,\ldots,m$. Taking \eqref{eq:S-localbasin} into account, we have $\ell\left(\mathcal{U}_{S}^{M^{\left(l+1\right)}}\right)=0$ for all $l=1,\ldots,m$ and $S<0$.
\end{proof}

More generally, concerning local stability indices for quasi-simple heteroclinic cycles, we have Theorem~\ref{th:negative_entries}. 
We assume now that some entry of at least one basic transition matrix is negative, which corresponds to a positive transverse eigenvalue at a node of the heteroclinic cycle.

We need a couple of auxiliary results and some more notation as follows. Define the set
\begin{equation}\label{U_R}
U_{R}\left(\boldsymbol{\alpha}_{1};\boldsymbol{\alpha}_{2};\ldots;\boldsymbol{\alpha}_{N}\right)=\left\{ \boldsymbol{y}\in U_{R}:\;\sum_{i=1}^{N}\alpha_{si}y_{i}<0 \textrm{ for } s=1, \hdots , N \right\}
\end{equation}
where $\boldsymbol{\alpha}_{s}=\left(\alpha_{s1},\ldots,\alpha_{sN}\right)\in\mathbb{R}^{N}$, 
$s=1,\ldots,N$ and $R<0$. 
Note that if $\alpha_{si}>0$ for all $s=1,\hdots,N$, then $\sum_{i=1}^{N}\alpha_{si}y_{i}<0$ trivially.

Consider $\boldsymbol{y}\in\mathbb{R}_{-}^{N}$ written as a linear combination of the eigenvectors of $M^{(j)}$ as in \eqref{eq:lincomb}. Then, 

\begin{equation}\label{eq:a} 
\boldsymbol{a} = \left( P^{(j)}\right)^{-1} \boldsymbol {y}.
\end{equation}
Let
$$
\boldsymbol{v}^{\max,j} = \left( v_{1}^{\max,j}, v_{2}^{\max,j}, \hdots , v_{N}^{\max,j} \right)
$$
be the line of $\left( P^{(j)}\right)^{-1}$ corresponding to the position associated with $\lambda_{\max}$. 
If $a_{\max}<0$ for all $\boldsymbol{y} \in \R^N_-$, then $U^{-\infty}\left(M^{\left(j\right)}\right)=\R^N_-$. Otherwise, using \eqref{eq:a} above and recalling that $U_0=\R^N_-$, we obtain
\begin{equation}\label{v_max}
U^{-\infty}\left(M^{(j)}\right)=\left\{ \boldsymbol{y}\in U_0:\;\sum_{i=1}^N v_{i}^{\max,j}y_i<0\right\} = U_0\left(\boldsymbol{v}^{\max,j};\boldsymbol{0};\ldots;\boldsymbol{0}\right).
\end{equation}

\begin{lemma}\label{lem:ell_positive}
Let $q=j_1,...,j_L$, $L\geq1$, denote all the indices for which $M_q$ has at least one negative entry.
Then, $\ell\left(U^{-\infty}\left(M^{\left(j\right)}\right)\right)>0$
for all $j=1,\ldots,m$, if and only if $\ell\left(U^{-\infty}\left(M^{\left(j\right)}\right)\right)>0$
for all $j=j_p+1$, $p=1,\ldots,L$, such that $j_p+1\notin \left\{ j_1,\ldots,j_L \right\}$.
\end{lemma}

\begin{proof}
The implication $\Rightarrow$ is trivial. 

For the implication $\Leftarrow$, notice that $\ell\left(U^{-\infty}\left(M^{\left(j\right)}\right)\right)>0$ if and only if $M^{\left(j\right)}$ satisfies conditions \emph{(i)-(iii)} of Lemma \ref{lemma5}. Since all matrices $M^{\left(j\right)}$ are similar, \emph{(i)-(ii)} hold simultaneously for all $j=1,\ldots,m$. 

Suppose for each $p=1,\ldots,L$ such that $j_p+1\notin \left\{ j_1,\ldots,j_L \right\}$ the matrix $M^{\left(j_p+1\right)}$ satisfies condition \emph{(iii)} of Lemma \ref{lemma5}. Without loss of generality, we assume that all entries of matrices $M_j$ with $j=j_p+1,j_p+2,\ldots,j_{p+1}-1$ are non-negative. As a consequence, the product $M_{\left(j-1,j_p+1\right)}=M_{j-1}\ldots M_{j_p+2}M_{j_p+1}$ has non-negative entries for any $j=j_p+2,\ldots,j_{p+1}$. On the other hand, 
$M_{\left(j-1,j_p+1\right)}M^{\left(j_p+1\right)}=M^{\left(j\right)}M_{\left(j-1,j_p+1\right)}$ and\footnote{\label{foot}If $\boldsymbol{w}$ is an eigenvector of $M^{(j)}$, that is, $M^{(j)}\boldsymbol{w} = \lambda \boldsymbol{w}$ for some scalar $\lambda$, then $M^{(l)}M_{(l-1,j)}\boldsymbol{w} = M_{(l-1,j)}M^{(j)}\boldsymbol{w}= \lambda M_{(l-1,j)}\boldsymbol{w}$. It means that, $M_{(l-1,j)}\boldsymbol{w}$ is an eigenvector for $M^{(l)}$ associated to the same eigenvalue $\lambda$.}
$$\boldsymbol{w}^{\max,j}=M_{\left(j-1,j_p+1\right)}\boldsymbol{w}^{\max,j_p+1}.$$
Therefore, if all components of $\boldsymbol{w}^{\max,j_p+1}$ have the same sign, then similarly 
all components of $\boldsymbol{w}^{\max,j}$ have the same sign for all remaining $j\in\cup_{p=1}^{L}\left\{j_p+2,\ldots,j_{p+1}\right\}$. 

\end{proof}

For each $l=1,\ldots,m$, denote by $\boldsymbol{\alpha}_s^l=\left(\alpha_{s1}^l, \hdots , \alpha_{sN}^l\right)$, $s=1,\ldots,N$, the $s^{th}$ row of the matrix $M_{(l,j)}$. Note that this guarantees that $\boldsymbol{\alpha}_s^l \neq \boldsymbol{0}$ for all $s=1,\ldots,N$.
\begin{lemma}\label{lem:cal_U_S}
Let $q=j_{1},...,j_{L}$, $L\geq1$, denote all the indices for which $M_q$ has at least one negative entry.
Suppose the matrices $M^{\left(j\right)}$ satisfy conditions (i)-(iii) of Lemma \ref{lemma5} for all $j=j_p+1$, $p=1,\ldots,L$, such that $j_p+1\notin \left\{ j_1,\ldots,j_L \right\}$. Then, $\boldsymbol{y}\in\mathbb{R}_{-}^{N}$ such that for all $l=1,\ldots,m$
\begin{equation}\label{limit}
\lim_{k\rightarrow\infty}M_{\left(l,j\right)}\left(M^{\left(j\right)}\right)^{k}\boldsymbol{y}=\boldsymbol{-\infty}
\end{equation}
is equivalent to 
$$
\boldsymbol{y}\in U^{-\infty}\left(M^{\left(j\right)}\right)\cap\left(\bigcap_{p=1}^{L}U_{0}\left(\boldsymbol{\alpha}_{1}^{j_{p}};\boldsymbol{\alpha}_{2}^{j_{p}};\ldots;\boldsymbol{\alpha}_{N}^{j_{p}}\right)\right).
$$
\end{lemma}

\begin{proof}
By chosing  $l=j-1$, we have $M_{\left(j-1,j\right)}\left(M^{(j)}\right)^k=\left(M^{(j)}\right)^{k+1}$ and the points $\boldsymbol{y}\in\mathbb{R}_{-}^{N}$ that satisfy
$\lim_{k\rightarrow\infty}\left(M^{(j)}\right)^{k+1}\boldsymbol{y}=\boldsymbol{-\infty}$ define the set
$U^{-\infty}\left(M^{(j)}\right)$ so that $\boldsymbol{y} \in U^{-\infty}\left(M^{(j)}\right)$.

On the other hand, we can write $M_{(l,j)}\left(M^{(j)}\right)^k=\left(M^{(l+1)}\right)^kM_{(l,j)}$ for any $l=1,\ldots,m$, $k\in\mathbb{N}_0$, and 
$$
\lim_{k\rightarrow\infty}M_{(l,j)}\left(M^{(j)}\right)^{k}\boldsymbol{y}=\lim_{k\rightarrow\infty}\left(M^{(l+1)}\right)^{k}\left[M_{(l,j)}\boldsymbol{y}\right].
$$
From \eqref{eq:basinlog2}, we have $\lim_{k\rightarrow\infty}\left(M^{(l+1)}\right)^{k}\left[M_{(l,j)}\boldsymbol{y}\right]=\boldsymbol{-\infty}$ if and only if 
$$
M_{(l,j)}\boldsymbol{y} \in U^{-\infty}\left(M^{(l+1)}\right).
$$
In order to describe the set of points for which this holds, we recall that conditions \emph{(i)-(iii)} of Lemma \ref{lemma5} ensure $\ell\left(U^{-\infty}\left(M\right)\right)>0$ where
\begin{equation}\label{Uinfty}
U^{-\infty}\left(M\right) = \left\{ \boldsymbol{y} \in \R^N_-: \; a_{\max}<0\right\}
\end{equation}
and $a_{\max}$ is the coefficient in front of $\boldsymbol{w}^{\max}$ in expansion \eqref{eq:lincomb} of $\boldsymbol{y}$ with respect to a matrix $M$. 
In particular, by virtue of Lemma \ref{lem:ell_positive},
$\ell\left(U^{-\infty}\left(M^{\left(j\right)}\right)\right)>0$ for all $j=1,\ldots,m$.

Considering $\boldsymbol {y} = P^{(j)}\boldsymbol{a}$, we have
$M_{(l,j)}\boldsymbol{y}=M_{(l,j)}P^{(j)}\boldsymbol{a}=P^{(l+1)}\boldsymbol{a}$ (see footnote \ref{foot}).
Then $a_{\max}$ is the same in the description of $\boldsymbol{y} \in U^{-\infty}\left(M^{(j)}\right)$ and of $M_{(l,j)}\boldsymbol{y} \in U^{-\infty}\left(M^{(l+1)}\right)$. Given \eqref{Uinfty}, if $\boldsymbol{y} \in U^{-\infty}\left(M^{(j)}\right)$, it suffices to demand that $M_{(l,j)}\boldsymbol{y} \in \R^N_-\equiv U_0$ so as to guarantee that $M_{(l,j)}\boldsymbol{y}$ belongs to $ U^{-\infty}\left(M^{(l+1)}\right)$.
It follows that
$$
M_{(l,j)}\boldsymbol{y} \in U_0 \Leftrightarrow \sum_{i=1}^N \alpha_{si}^l y_i < 0 \;\;\; \textrm{ for all }s=1,\hdots ,N,
$$
which is equivalent to $\boldsymbol{y} \in U_0\left(\boldsymbol{\alpha}_{1}^l, \hdots , \boldsymbol{\alpha}_{N}^l\right)$ according to \eqref{U_R}. 

We do not need to consider all the $m$ sets $U_0\left(\boldsymbol{\alpha}_{1}^l, \hdots , \boldsymbol{\alpha}_{N}^l\right)$. In fact, when $M_{\left(q,j\right)}$ is the product of matrices with non-negative entries, we have that
$\lim_{k\rightarrow\infty}M_{\left(q,j\right)}\left(M^{\left(j\right)}\right)^{k}\boldsymbol{y}=\boldsymbol{-\infty}$
holds for any $\boldsymbol{y} \in U^{-\infty}\left(M^{(j)}\right)$.
Suppose the basic transition matrix $M_{q+1}$, such that $M_{q+1}M_{\left(q,j\right)}=M_{\left(q+1,j\right)}$, has at least one negative entry. Then, as above,
$\lim_{k\rightarrow\infty}M_{\left(q+1,j\right)}\left(M^{\left(j\right)}\right)^{k}\boldsymbol{y}=\boldsymbol{-\infty}$
for any $\boldsymbol{y} \in U^{-\infty}\left(M^{(j)}\right) \cap U_0\left(\boldsymbol{\alpha}_{1}^{q+1}, \hdots , \boldsymbol{\alpha}_{N}^{q+1}\right)$. Thus, the set of points that satisfy \eqref{limit} for all $l=1,\ldots,m$ gets restricted to
$$
U^{-\infty}\left(M^{\left(j\right)}\right)\cap\left(\bigcap_{p=1}^{L}U_{0}\left(\boldsymbol{\alpha}_{1}^{j_{p}};\boldsymbol{\alpha}_{2}^{j_{p}};\ldots;\boldsymbol{\alpha}_{N}^{j_{p}}\right)\right).
$$
\end{proof}

The calculations required to obtain the local stability index are performed by means of the function $F^{\textrm{index}}$, the analogue of the function $f^{\textrm{index}}$ of \cite[p. 905]{PodviginaAshwin2011}.

\begin{definition}\label{def:F_index}
In coordinates \eqref{eq:change}, let the intersection of the local basin of attraction of a compact invariant set $X$ with a cross-section transverse to the flow at $x \in X$ be given by $U_R\left(\boldsymbol{\alpha};\boldsymbol{0};\ldots;\boldsymbol{0}\right)$ for some $\boldsymbol{\alpha}=\left(\alpha_1,\alpha_2,\ldots,\alpha_N\right) \in \R^N$ and $R<0$. We define the function $\boldsymbol{\alpha} \mapsto F^{\textrm{index}}\left(\boldsymbol{\alpha}\right)$ to be the local stability index for $X$ at $x$ relative to this intersection, i.e. $F^{\textrm{index}}\left(\boldsymbol{\alpha}\right)=\sigma_{\textrm{loc}}(x)$.
\end{definition}

Lemma~\ref{lem:cal_U_S} and Theorem~\ref{th:negative_entries} show that the set $\mathcal{B}_{\delta}^{\pi_j}$ in $H^{in}_{j}$, in coordinates \eqref{eq:change}, is of the form 
$U_{R}\left(\boldsymbol{\alpha}_{1};\boldsymbol{\alpha}_{2};\ldots;\boldsymbol{\alpha}_{N}\right)$. By virtue of \eqref{U_R}, we can write 
$$
U_R\left(\boldsymbol{\alpha}_1;\boldsymbol{\alpha}_2;\ldots;\boldsymbol{\alpha}_N\right)=\bigcap_{i=1}^{N}U_R\left(\boldsymbol{\alpha}_i;\boldsymbol{0};\ldots;\boldsymbol{0}\right).
$$
Thus, the stability index for the connection that intersects $H^{in}_j$, characterizing the local geometry of $\B_{\delta}^{\pi_j}$, is given by
\[
\min_{i=1,\ldots,N}\left\{ F^{\textrm{index}}\left(\boldsymbol{\alpha}_{i}\right)\right\} .
\]

The function $F^{\textrm{index}}$ is used to determine the local stability index at a point of a heteroclinic cycle through two components, $F^-$ and $F^+$, related to $\sigma_{\textnormal{loc},-}$ and $\sigma_{\textnormal{loc},+}$ respectively, in Definition \ref{def:stab_index}.

\begin{lemma}\label{lem:Findex}
Let 
$\boldsymbol{\alpha}=\left(\alpha_{1},\ldots,\alpha_{N}\right)
\in\mathbb{R}^{N}$
and $\alpha_{\min}=\min_{i=1,\ldots,N}\alpha_{i}$. The values of
the function $F^{\textrm{index}}:\mathbb{R}^N\rightarrow\mathbb{R}$
are
\[
F^{\textrm{index}}\left(\boldsymbol{\alpha}\right)=F^{+}\left(\boldsymbol{\alpha}\right)-F^{-}\left(\boldsymbol{\alpha}\right)
\]
where $F^{-}\left(\boldsymbol{\alpha}\right)=F^{+}\left(-\boldsymbol{\alpha}\right)$ and
\[
F^{+}\left(\boldsymbol{\alpha}\right)=
\begin{cases}
+\infty, & \textrm{if }\alpha_{\min}\geq0 \\
0, & \textrm{if }\sum_{i=1}^{N}\alpha_{i}\leq0 \\
-\frac{1}{\alpha_{\min}}\sum_{i=1}^{N}\alpha_{i} ,& \textrm{if }\alpha_{\min}<0\textrm{ and }\sum_{i=1}^{N}\alpha_{i}\geq0.
\end{cases}
\]
\end{lemma}

See Appendix \ref{sec:Findex} for the proof of Lemma~\ref{lem:Findex} as well as the explicit values of $F^{\textrm{index}}$ when $N=3$.

We can now study the stability of heteroclinic cycles when the basic transition matrices have negative entries. The proof of Theorem~\ref{th:negative_entries} is constructive and provides an expression for the stability indices in terms of the function $F^{\textrm{index}}$.

\begin{theorem}\label{th:negative_entries}
Let $M_{j}$, $j=1,\ldots,m$, be  basic transition matrices of a collection of maps
associated with a heteroclinic cycle. Denote by $q=j_{1},...,j_{L}$, $L\geq1$, all the indices for which $M_q$ has at least one negative entry.
\begin{enumerate}
\item [(a)] If, for at least one $j$, the matrix $M^{\left(j\right)}$
does not satisfy conditions (i)-(iii) of Lemma \ref{lemma5},
then $\sigma_{j}=-\infty$ for all $j=1,\ldots,m$ and the cycle is
not an attractor.
\item [(b)] If the matrices $M^{\left(j\right)}$ satisfies conditions (i)-(iii)
of Lemma \ref{lemma5} for all $j=j_p+1$, $p=1,\ldots,L$, such that $j_p+1\notin \left\{ j_1,\ldots,j_L \right\}$,
then the cycle is f.a.s. Furthermore, for each $j=1,\ldots,m$, there exist vectors $\boldsymbol{\beta}_{1},\boldsymbol{\beta}_{2},\ldots,\boldsymbol{\beta}_{K}\in\R^N$,
such that 
\[
\sigma_{j}=\min_{i=1,\ldots,K}\left\{ F^{\textrm{index}}\left(\boldsymbol{\beta}_{i}\right)\right\} .
\]
\end{enumerate}
\end{theorem}

\begin{proof}
We treat each case separately.
\begin{enumerate}
\item[(a)] This follows immediately from Lemma \ref{lem:l(U(M))}.

\item[(b)]
Using Lemma \ref{lem:ell_positive}, we have $\ell\left(U^{-\infty}\left(M^{\left(j\right)}\right)\right)>0$ for all $j=1,\ldots,m$. Then the cycle is f.a.s.

By Definition \ref{def:stab_index}, the local stability index is computed for $\delta>0$ small, that is, for large negative $S=\ln\left(\delta\right)$.
Given \eqref{eq:S-localbasin},
we can write
\[
\begin{aligned}\mathcal{U}_{S}^{M^{\left(j\right)}}= & U_{0}\left(\boldsymbol{v}^{\max,j};\boldsymbol{0};\ldots;\boldsymbol{0}\right)\cap\left(\bigcap_{p=1}^{L}U_{0}\left(\boldsymbol{\alpha}_{1}^{j_{p}};\boldsymbol{\alpha}_{2}^{j_{p}};\ldots;\boldsymbol{\alpha}_{N}^{j_{p}}\right)\right)\\
 \cap&\left\{ \boldsymbol{y}\in\mathbb{R}_{-}^{N}:\;M_{\left(l,j\right)}\left(M^{\left(j\right)}\right)^{k}\boldsymbol{y}\in U_{S}\textrm{ for all } l=1,\ldots,m,\;k\in\mathbb{N}_{0}\right\}
\end{aligned}
\]
from \eqref{v_max} together with Lemma~\ref{lem:cal_U_S}.
As $S$ approaches $-\infty$, the condition defining the third set above becomes redundant in the sense that it holds true for all $\boldsymbol{y}\in U_R$ with $R<S$.
Therefore, $U_{R}\cap\mathcal{U}_{S}^{M^{\left(j\right)}}$ is reduced to
$$
U_{R}\left(\boldsymbol{v}^{\max,j};\boldsymbol{0};\ldots;\boldsymbol{0}\right)\cap\left(\bigcap_{p=1}^{L}U_{R}\left(\boldsymbol{\alpha}_{1}^{j_{p}};\boldsymbol{\alpha}_{2}^{j_{p}};\ldots;\boldsymbol{\alpha}_{N}^{j_{p}}\right)\right).
$$

The local stability index relative to $U_{R}\left(\boldsymbol{v}^{\max,j};\boldsymbol{0};\ldots;\boldsymbol{0}\right)$
is $F^{\textrm{index}}\left(\boldsymbol{v}^{\max,j}\right)$.
Similarly, the local stability index relative to $U_{R}\left(\boldsymbol{\alpha}_{1}^{j_{p}};\boldsymbol{\alpha}_{2}^{j_{p}};\ldots;\boldsymbol{\alpha}_{N}^{j_{p}}\right)$ is given by
$\min_{i=1,\ldots,N}F^{\textrm{index}}\left(\boldsymbol{\alpha}_{i}^{j_{p}}\right)$ for each $p=1,\ldots,L$ so that the local stability index relative to $\bigcap_{p=1}^{L}U_{R}\left(\boldsymbol{\alpha}_{1}^{j_{p}};\boldsymbol{\alpha}_{2}^{j_{p}};\ldots;\boldsymbol{\alpha}_{N}^{j_{p}}\right)$
is 
$$
\min_{p=1,\ldots,L}\left\{ \min_{i=1,\ldots,N}F^{\textrm{index}}\left(\boldsymbol{\alpha}_{i}^{j_{p}}\right)\right\}.
$$
Thus,
$$
\sigma_{j}= \min\left\{ F^{\textrm{index}}\left(\boldsymbol{v}^{\max,j}\right),\min_{p=1,\ldots,L}\left\{ \min_{i=1,\ldots,N}\left\{ F^{\textrm{index}}\left(\boldsymbol{\alpha}_{i}^{j_{p}}\right)\right\} \right\} \right\} .
$$
\end{enumerate}
\end{proof}

Note that when $U^{-\infty}(M^{(j)})=\R^N_-$ it follows from \eqref{v_max} that $F^{\textrm{index}}\left(\boldsymbol{v}^{\max,j}\right)=+\infty$ and the above expression simplifies. 

\section{A cycle in the Rock-Scissors-Paper game}\label{sec:RSP}

The Rock-Scissors-Paper (RSP) game is well-known: there are two players, $X$ and $Y$. Each player chooses among three actions ($R$, for Rock, $S$, for Scissors, and $P$, for Paper) knowing that $R$ beats $S$, $S$ beats $P$ and $P$ beats $R$. The winning player receives $+1$ while the losing player receives $-1$. When both players choose the same action, there is a draw. We use the description of Sato {\em et al.} \cite{SAC} and attribute $\varepsilon_x, \varepsilon_y \in (-1,1)$ to player $X$ and to player $Y$, respectively, for a drawing action. Under repetition, this game can be described by a dynamical system on a 4-dimensional state space contained in $\R^6$. The state variables are the probabilities of playing each of the actions. We then associate to player $X$ the state vector $\boldsymbol{x}=\left(x_{1},x_{2},x_{3}\right) \in \Delta_X$,
with $x_{1},x_{2},x_{3}\geq0$ and $x_{1}+x_{2}+x_{3}=1$, whose components are the probabilities of player $X$ playing $R$, $S$ and $P$, respectively. Analogously, for player $Y$, we have $\boldsymbol{y}=\left(y_{1},y_{2},y_{3}\right) \in \Delta_Y$ with $y_{1},y_{2},y_{3}\geq0$ and $y_{1}+y_{2}+y_{3}=1$. We denote the colective state space by $\Delta = \Delta_X \times \Delta_Y \subset \R^6$, which is the product of two $2$-dimensional simplices.

The dynamics for this problem are detailed in \cite{AC} and 
\cite{SAC}. The vector field is equivariant by the group $\Gamma = \Z_3$ generated by
$$
(x_1,x_2,x_3;y_1,y_2,y_3) \mapsto (x_3,x_1,x_2;y_3,y_1,y_2).
$$
Under the action of $\Gamma$, there are three group orbits of equilibria
\[
\begin{aligned}\xi_{0}\equiv\varGamma\left(R,P\right)= & \left\{ \left(R,P\right),\left(S,R\right),\left(P,S\right)\right\} \\
\xi_{1}\equiv\varGamma\left(R,S\right)= & \left\{ \left(R,S\right),\left(S,P\right),\left(P,R\right)\right\} \\
\xi_{2}\equiv\varGamma\left(R,R\right)= & \left\{ \left(R,R\right),\left(S,S\right),\left(P,P\right)\right\} ,
\end{aligned}
\]
where each element corresponds to a vertex in $\Delta$. Together with all connections $[\xi_i \rightarrow \xi_j]$, $i\neq j=0,1,2$, these group orbits of equilibria comprise a quotient heteroclinic network, see \cite{AC}.

We apply the methods\footnote{A presentation of the full detail of the calculations involved in this section would not be helpful as an illustration, besides taking us beyond the scope of this article. The full detail may be found in \cite{GSC_RSP}.} of Section \ref{sec:stab_cycles} to calculate the local stability indices along $[\xi_0 \rightarrow \xi_1]$ and $[\xi_1 \rightarrow \xi_0]$. Such connections are $1$-dimensional and constitute the cycle $C_0$ in \cite{AC}. Note that $C_0$ is not simple: although the connections are contained in 2-dimensional invariant spaces, these are not fixed-point spaces. Definition~\ref{quasi-simple} in turn holds so that $C_0$ is quasi-simple. Following \cite[Remark p.\ 1603]{AC}, we assume the global maps to be the identity\footnote{Even though the global maps are the identity there is a permutation when going from one local basis to the next in consecutive cross-sections. In \cite{AC} these permutations were overlooked. See the comments at the end of this section.}, ensuring the validity of Assumption A. 

In appropriate coordinates \eqref{eq:change}, the basic transition matrices for this cycle are
\[
M_{0} =\left[\begin{array}{ccc}
\dfrac{1-\varepsilon_y}{2} & 1 & 0\\
 & & \\
-\dfrac{1+\varepsilon_x}{2} & 0 & 1\\
 & & \\
1 & 0 & 0
\end{array}\right], \; \; \;  \;
M_{1} =\left[\begin{array}{ccc}
\dfrac{1-\varepsilon_x}{2} & 1 & 0\\
 & & \\
-\dfrac{1+\varepsilon_y}{2} & 0 & 1\\
 & & \\
1 & 0 & 0
\end{array}\right].
\]
Conditions \emph{(i)-(iii)} of Lemma \ref{lemma5} hold true for the products $M^{(0)}=M_1 M_0$ and $M^{(1)}=M_0 M_1$ if and only if $\varepsilon_x+\varepsilon_y<0$.  
In case (b) of Theorem~\ref{th:negative_entries} we make use of the function
$F^{\textrm{index}}$ in Appendix \ref{app:F_index3} to obtain $\sigma_{0}=\sigma_{\textrm{loc}}\left[\xi_{1}\rightarrow\xi_{0}\right]$ and $\sigma_{1}=\sigma_{\textrm{loc}}\left[\xi_{0}\rightarrow\xi_{1}\right]$.

\begin{lemma}\label{stab_RSP}
Assuming that $\varepsilon_x+\varepsilon_y<0$,  the local stability indices  for the cycle $C_0$ of the RSP game are finite and given by
$$
\sigma_{0}=\min\left\{ \frac{1-\varepsilon_{x}}{1+\varepsilon_{x}}, \frac{(1-\varepsilon_y)^2}{2(1+\varepsilon_y)}\right\}>0, \;\;\;\; \sigma_{1}=\min\left\{  \frac{1-\varepsilon_{y}}{1+\varepsilon_{y}}, \frac{(1-\varepsilon_x)^2}{2(1+\varepsilon_x)} \right\}>0.
$$
\end{lemma}

\begin{proof}
We can show that $U^{-\infty}(M^{(0)})=U^{-\infty}(M^{(1)})=\R^3_-$. Therefore, the objects of the function $F^{\textrm{index}}$ are the rows of $M_0$ and of $M^{(0)}$ for the calculation of $\sigma_0$ and the rows of $M_1$ and of $M^{(1)}$ for that of $\sigma_1$. Actually, applying Theorem~\ref{th:negative_entries}, we have 
\begin{eqnarray*}
\sigma_{0} & = & \min \left\{ F^{\textrm{index}}\left(\dfrac{1-\varepsilon_{y}}{2},1,0\right),F^{\textrm{index}}\left(-\dfrac{1+\varepsilon_{x}}{2},0,1\right), F^{\textrm{index}}\left(1,0,0\right),\right.\\
 & & \left. F^{\textrm{index}}\left(\dfrac{-1-3\varepsilon_x-\varepsilon_y+\varepsilon_x\varepsilon_y}{4}, \dfrac{1-\varepsilon_x}{2}, 1\right),
 F^{\textrm{index}}\left(\dfrac{3+\varepsilon_y^2}{4}, -\dfrac{1+\varepsilon_y}{2},0\right) \right\}.
\end{eqnarray*}

Using Lemma~\ref{lem:Findex} we get $F^{\textrm{index}}\left(\frac{1-\varepsilon_{y}}{2},1,0\right)=F^{\textrm{index}}\left(1,0,0\right)=+\infty$. Since 
$$
F^{\textrm{index}}\left(\dfrac{-1-3\varepsilon_x-\varepsilon_y+\varepsilon_x\varepsilon_y}{4}, \dfrac{1-\varepsilon_x}{2}, 1\right) >  F^{\textrm{index}}\left(-\dfrac{1+\varepsilon_{x}}{2},0,1\right)
$$
the minimum is between
$$
F^{\textrm{index}}\left(-\dfrac{1+\varepsilon_{x}}{2},1,0\right)=-\dfrac{1}{-\frac{1+\varepsilon_{x}}{2}}\left(-\dfrac{1+\varepsilon_{x}}{2}+1\right)=\dfrac{1-\varepsilon_{x}}{1+\varepsilon_{x}}>0
$$
and 
$$
F^{\textrm{index}}\left(\dfrac{3+\varepsilon_y^2}{4}, -\dfrac{1+\varepsilon_y}{2},0\right)=\dfrac{(1-\varepsilon_y)^2}{2(1+\varepsilon_y)}>0.
$$
Concerning $\sigma_1$, we have
\begin{eqnarray*}
\sigma_{1} & = & \min \left\{ F^{\textrm{index}}\left(\dfrac{1-\varepsilon_{x}}{2},1,0\right),F^{\textrm{index}}\left(-\dfrac{1+\varepsilon_{y}}{2},0,1\right), F^{\textrm{index}}\left(1,0,0\right),\right.\\
 & & \left. F^{\textrm{index}}\left(\dfrac{-1-3\varepsilon_y-\varepsilon_x+\varepsilon_y\varepsilon_x}{4}, \dfrac{1-\varepsilon_y}{2}, 1\right),
 F^{\textrm{index}}\left(\dfrac{3+\varepsilon_x^2}{4}, -\dfrac{1+\varepsilon_x}{2},0\right) \right\}.
\end{eqnarray*}

The result follows easily by noting that the quantities in $\sigma_1$ are obtained from those in $\sigma_0$ by interchanging $\varepsilon_x$ and $\varepsilon_y$. 
\end{proof}

Thus it follows that the cycle $C_0$ in the RSP game is e.a.s. if and only if $\varepsilon_x+\varepsilon_y<0$. 

We end with two brief comments concerning Lemma 4.9 in \cite{AC} that states that none of the cycles in the RSP network is e.a.s., and which must be incorrect in  view of Lemma~\ref{stab_RSP} above. The first comment is that the stability index provides a very powerful and systematic tool for the study of stability and it was not available at the time of \cite{AC}. The second is that the authors of \cite{AC} forgot to take into account the permutation between consecutive cross-sections. 

On a positive note, the e.a.s. of the cycle $C_0$ of the RSP game is in accordance with the numerical simulations of \cite{SAC}, which show this cycle as possessing some attracting properties, see \cite[Figure 15, top]{SAC}.

\paragraph{Acknowledgements:} We are grateful to A.\ Lohse for assistance in the calculation of the 3-variable case of the function $F^{\textrm{index}}$, and to A.\ Rodrigues for fruitful discussions.
We thank an anonymous referee of a previous version for some very pertinent comments.

Both authors were partially supported by CMUP (UID/MAT/00144/2013), which is funded by FCT (Portugal) with national (MEC) and European structural funds (FEDER), under the partnership agreement PT2020. L.\ Garrido-da-Silva is the recipient of the doctoral grant PD/BD/105731/2014 from FCT (Portugal).

\appendix
\section{The function $F^{\textrm{index}}$\label{sec:Findex}}

We prove Lemma \ref{lem:Findex}. Consider a return map $\pi \equiv \pi_1:\mathbb{R}^{N}\rightarrow\mathbb{R}^{N}$ \eqref{pi_map} associated with a point on a heteroclinic connection that intersects $H^{in}_1$. For small $\delta>0$, let $\mathcal{B}_{\delta}^{\pi}$ be the $\delta$-local basin of attraction of $\boldsymbol{0}\in\mathbb{R}^{N}$ for the map $\pi$.

Given $\boldsymbol{\alpha}=\left(\alpha_{1},\alpha_{2},\ldots,\alpha_{N}\right)\in\mathbb{R}^{N}$, suppose that $\mathcal{B}_{\delta}^{\pi}$ in $H^{in}_1$ in the new coordinates \eqref{eq:change} is $U_{R}\left(\boldsymbol{\alpha};\boldsymbol{0};\ldots;\boldsymbol{0}\right)$ for some large $R<0$. 
Since the measure of a set is always regarded as its measure in the original variables, we represent the latter set in original coordinates $\left(w,\boldsymbol{z}\right)=\left(x_1,x_2,\ldots,x_N\right)\equiv\boldsymbol{x}$ as $\tilde{U}_{R}\left(\boldsymbol{\alpha}\right)$. Then,
\begin{equation}\label{eq:tilde_U}
\tilde{U}_{R}\left(\boldsymbol{\alpha}\right)=\left\{ \boldsymbol{x}\in \tilde{U}_R:\left|x_{1}^{\alpha_{1}}x_{2}^{\alpha_{2}}\cdots x_{N}^{\alpha_{N}}\right|<1\right\}.
\end{equation}
where
$$
\tilde{U}_R=\left\{ \boldsymbol{x}:\textrm{ }\max_{i}\textrm{ }\left|x_{i}\right|<\textrm{e}^R\right\}.
$$ 
In particular, the set $\tilde{U}_R$ is an open ball of radius  $\varepsilon=\textrm{e}^R>0$ centred at $\boldsymbol{0}\in\mathbb{R}^{N}$ for the maximum norm. Hence, by virtue of Definition \ref{def:F_index}, the value of $F^{\textrm{index}}\left(\boldsymbol{\alpha}\right)=F^{+}\left(\boldsymbol{\alpha}\right)-F^{-}\left(\boldsymbol{\alpha}\right)$ quantifies the local extent of $\tilde{U}_{R}\left(\boldsymbol{\alpha}\right)$ such that 
\[
F^{-}\left(\boldsymbol{\alpha}\right)=\lim_{R\rightarrow-\infty}\frac{\ln\left(\Sigma_{R}\right)}{R},\qquad F^{+}\left(\boldsymbol{\alpha}\right)=\lim_{R\rightarrow-\infty}\frac{\ln\left(1-\Sigma_{R}\right)}{R}
\]
with
\[
\Sigma_{R}=\frac{\ell\left(\tilde{U}_{R}\left(\boldsymbol{\alpha}\right)\right)}{\ell\left(\tilde{U}_R\right)}.
\]

Notice that the region in $\mathbb{R}^{N}$ determined by the whole $\tilde{U}_{R}\left(\boldsymbol{\alpha}\right)$ is invariant with respect to the reflections $x_{j}\mapsto-x_{j}$, $j=1,\ldots,N$. Thus, we only consider $\boldsymbol{x}\in\mathbb{R}_{+}^{N}$.
Define $\alpha_{\min}=\min_{i=1,\ldots,N}\alpha_{i}$. The construction of $F^{+}$ proceeds in three cases. 

\paragraph{Case 1: } 

If $\alpha_{\min}\geq0$, then $\alpha_{i}\geq0$ for all $i=1,\ldots,N$.
Given \eqref{eq:tilde_U}, we have $\alpha_i >0$ for at least one $i=1,\ldots,N$. Then, for sufficiently large negative $R$, $\tilde{U}_{R}\left(\boldsymbol{\alpha}\right)$ is reduced to $\tilde{U}_R$. Thus, $\Sigma_R=1$ and $F^{+}\left(\boldsymbol{\alpha}\right)=+\infty$. 

\paragraph{Case 2:}

If $\alpha_{\min}<0$ and $\sum_{i=1}^{N}\alpha_{i}>0$, then there exist $K\geq1$ and $i_{1},\ldots,i_{K}\in\left\{ 1,\ldots,N\right\} $ such that $\alpha_{i_{1}},\ldots,\alpha_{i_{K}}<0$ and $\alpha_{i}\geq0$ for any $i\neq i_{\text{1}},\ldots,i_{K}$ or $\alpha_{i_{1}},\ldots,\alpha_{i_{K}}\geq0$ and $\alpha_{i}<0$ for any $i\neq i_{\text{1}},\ldots,i_{K}$. We
provide the proof when $K=i_{K}=1$ as the technique remains the same but the calculations are unnecessarily complicated.

Suppose first that $\alpha_{1}<0$, $\alpha_{i}\geq0$ for all $i=2,\ldots,N$. Then $\alpha_{\min}=\alpha_{1}$. Given 
$\boldsymbol{x}\in\tilde{U}_{R}\left(\boldsymbol{\alpha}\right)$, we have $\boldsymbol{x}\in\tilde{U}_R$ and
\[
\begin{aligned}x_{1} & >x_{2}^{-\frac{\alpha_{2}}{\alpha_{1}}}x_{3}^{-\frac{\alpha_{3}}{\alpha_{1}}}\cdots x_{N}^{-\frac{\alpha_{N}}{\alpha_{1}}}\\
x_{i} & <x_{i_{1}}^{-\frac{\alpha_{i_{1}}}{\alpha_{i}}}x_{i_{2}}^{-\frac{\alpha_{i_{2}}}{\alpha_{i}}}\cdots x_{i_{N-1}}^{-\frac{\alpha_{i_{N-1}}}{\alpha_{i}}}, &  & i=2,\ldots,N;\textrm{ }i\neq i_{l},\textrm{ }l=1,\ldots,N-1,
\end{aligned}
\]
so that 
\[
\begin{aligned}\ell\left(\left.\tilde{U}_R\right\backslash \tilde{U}_{R}\left(\boldsymbol{\alpha}\right)\right) & =\int_{0}^{\varepsilon}\int_{0}^{\varepsilon}\cdots\int_{0}^{\varepsilon}\int_{0}^{x_{2}^{-\frac{\alpha_{2}}{\alpha_{1}}}x_{3}^{-\frac{\alpha_{3}}{\alpha_{1}}}\cdots x_{N}^{-\frac{\alpha_{N}}{\alpha_{1}}}}dx_{1}dx_{2}\ldots dx_{N-1}dx_{N}\\
 & =\frac{\varepsilon^{N-\frac{1}{\alpha_{1}}\sum_{i=1}^{N}\alpha_{i}}}{\prod_{i=2}^{N}\left(-\frac{\alpha_{i}}{\alpha_{1}}+1\right)}
\end{aligned}
\]
and
\[
1-\Sigma_{R}=\frac{\ell\left(\left.\tilde{U}_R\right\backslash \tilde{U}_{R}\left(\boldsymbol{\alpha}\right)\right)}{\ell\left(\tilde{U}_R\right)}=\frac{\frac{\varepsilon^{N-\frac{1}{\alpha_{1}}\sum_{i=1}^{N}\alpha_{i}}}{\prod_{i=2}^{N}\left(-\frac{\alpha_{i}}{\alpha_{1}}+1\right)}}{\varepsilon^{N}}=\frac{\varepsilon^{-\frac{1}{\alpha_{1}}\sum_{i=1}^{N}\alpha_{i}}}{\prod_{i=2}^{N}\left(-\frac{\alpha_{i}}{\alpha_{1}}+1\right)}.
\]
Therefore
\[
\frac{\ln\left(1-\Sigma_{R}\right)}{R}=-\frac{1}{\alpha_{1}}\sum_{i=1}^{N}\alpha_{i}-\frac{\ln\left(\prod_{i=2}^{N}\left(-\frac{\alpha_{i}}{\alpha_{1}}+1\right)\right)}{R}
\]
and 
\[
F^{+}\left(\boldsymbol{\alpha}\right)=-\frac{1}{\alpha_{1}}\sum_{i=1}^{N}\alpha_{i}-\lim_{R\rightarrow-\infty}\frac{\ln\left(\prod_{i=2}^{N}\left(-\frac{\alpha_{i}}{\alpha_{1}}+1\right)\right)}{R}
=-\frac{1}{\alpha_{\min}}\sum_{i=1}^{N}\alpha_{i}>0.
\]

Suppose now that $\alpha_{1}\geq0$, $\alpha_{i}<0$ for all $i=2,\ldots,N$. 
If $\boldsymbol{x}\in\tilde{U}_R\left(\boldsymbol{\alpha}\right)$, then $\boldsymbol{x}\in\tilde{U}_R$ and
\[
\begin{aligned}x_{1} & <x_{2}^{-\frac{\alpha_{2}}{\alpha_{1}}}x_{3}^{-\frac{\alpha_{3}}{\alpha_{1}}}\cdots x_{N}^{-\frac{\alpha_{N}}{\alpha_{1}}}\\
x_{i} & >x_{i_{1}}^{-\frac{\alpha_{i_{1}}}{\alpha_{i}}}x_{i_{2}}^{-\frac{\alpha_{i_{2}}}{\alpha_{i}}}\cdots x_{i_{N-1}}^{-\frac{\alpha_{i_{N-1}}}{\alpha_{i}}}, &  & i=2,\ldots,N;\textrm{ }i\neq i_{l},\textrm{ }l=1,\ldots,N-1.
\end{aligned}
\]
In calculating the measure of the set determined by these inequalities, we must take into account how the respective boundaries intersect the boundaries of $\tilde{U}_R$.
These intersections lead to a splitting of the integral that expresses the measure of $\tilde{U}_R \backslash \tilde{U}_{R}\left(\boldsymbol{\alpha}\right)$
(and of $\tilde{U}_R\left(\boldsymbol{\alpha}\right)$) which was not necessary before. Accordingly, the integrals that follow have a rather daunting aspect. Indeed,
\[
\begin{aligned} & \ell\left(\left.\tilde{U}_R\right\backslash \tilde{U}_{R}\left(\boldsymbol{\alpha}\right)\right)=\\
= & \int_{0}^{\varepsilon}\int_{0}^{\varepsilon}\int_{0}^{\min\left\{ \varepsilon,\varepsilon^{-\sum_{i=4}^{N}\frac{\alpha_{i}}{\alpha_{3}}}x_{1}^{-\frac{\alpha_{1}}{\alpha_{3}}}x_{2}^{-\frac{\alpha_{2}}{\alpha_{3}}}\right\} }\int_{0}^{\min\left\{ \varepsilon,\varepsilon^{-\sum_{i=5}^{N}\frac{\alpha_{i}}{\alpha_{4}}}x_{1}^{-\frac{\alpha_{1}}{\alpha_{4}}}x_{2}^{-\frac{\alpha_{2}}{\alpha_{4}}}x_{3}^{-\frac{\alpha_{3}}{\alpha_{4}}}\right\} }\cdots\\
 & \cdots\int_{0}^{\min\left\{ \varepsilon,\varepsilon^{-\frac{\alpha_{N}}{\alpha_{N-1}}}x_{1}^{-\frac{\alpha_{1}}{\alpha_{N-1}}}\ldots x_{N-2}^{-\frac{\alpha_{N-2}}{\alpha_{N-1}}}\right\} }\int_{0}^{\min\left\{ \varepsilon,x_{1}^{-\frac{\alpha_{1}}{\alpha_{N}}}\ldots x_{N-1}^{-\frac{\alpha_{N-1}}{\alpha_{N}}}\right\} }dx_{N}dx_{N-1}\ldots dx_{4}dx_{3}dx_{2}dx_{1}\\
= & \int_{0}^{\varepsilon}\left\{ \int_{0}^{\varepsilon^{-\sum_{i=3}^{N}\frac{\alpha_{i}}{\alpha_{2}}}x_{1}^{-\frac{\alpha_{1}}{\alpha_{2}}}}\int_{0}^{\varepsilon}\cdots\int_{0}^{\varepsilon}dx_{N}\ldots dx_{3}dx_{2}+\right.\\
 & +\int_{\varepsilon^{-\sum_{i=3}^{N}\frac{\alpha_{i}}{\alpha_{2}}}x_{1}^{-\frac{\alpha_{1}}{\alpha_{2}}}}^{\varepsilon}\int_{0}^{\varepsilon^{-\sum_{i=4}^{N}\frac{\alpha_{i}}{\alpha_{3}}}x_{1}^{-\frac{\alpha_{1}}{\alpha_{3}}}x_{2}^{-\frac{\alpha_{2}}{\alpha_{3}}}}\int_{0}^{\varepsilon}\cdots\int_{0}^{\varepsilon}dx_{N}\ldots dx_{4}dx_{3}dx_{2}+\\
 & \begin{aligned}+\int_{\varepsilon^{-\sum_{i=3}^{N}\frac{\alpha_{i}}{\alpha_{2}}}x_{1}^{-\frac{\alpha_{1}}{\alpha_{2}}}}^{\varepsilon}\int_{\varepsilon^{-\sum_{i=4}^{N}\frac{\alpha_{i}}{\alpha_{3}}}x_{1}^{-\frac{\alpha_{1}}{\alpha_{3}}}x_{2}^{-\frac{\alpha_{2}}{\alpha_{3}}}}^{\varepsilon} & \int_{0}^{\varepsilon^{-\sum_{i=5}^{N}\frac{\alpha_{i}}{\alpha_{4}}}x_{1}^{-\frac{\alpha_{1}}{\alpha_{4}}}x_{2}^{-\frac{\alpha_{2}}{\alpha_{4}}}x_{3}^{-\frac{\alpha_{3}}{\alpha_{4}}}}\\
 & \int_{0}^{\varepsilon}\cdots\int_{0}^{\varepsilon}dx_{N}\ldots dx_{5}dx_{4}dx_{3}dx_{2}+
\end{aligned}
\\
 &+\cdots+\\
\\
 & \begin{aligned}+\int_{\varepsilon^{-\sum_{i=3}^{N}\frac{\alpha_{i}}{\alpha_{2}}}x_{1}^{-\frac{\alpha_{1}}{\alpha_{2}}}}^{\varepsilon}\ldots & \int_{\varepsilon^{-\sum_{i=N-1}^{N}\frac{\alpha_{i}}{\alpha_{N-2}}}x_{1}^{-\frac{\alpha_{1}}{\alpha_{N-2}}}\ldots x_{N-3}^{-\frac{\alpha_{N-3}}{\alpha_{N-2}}}}^{\varepsilon}\\
 & \int_{0}^{\varepsilon^{-\frac{\alpha_{N}}{\alpha_{N-1}}}x_{1}^{-\frac{\alpha_{1}}{\alpha_{N-1}}}\ldots x_{N-2}^{-\frac{\alpha_{N-2}}{\alpha_{N-1}}}}\int_{0}^{\varepsilon}dx_{N}dx_{N-1}dx_{N-2}\ldots dx_{2}+
\end{aligned}
\\
 & \begin{aligned}+ & \int_{\varepsilon^{-\sum_{i=3}^{N}\frac{\alpha_{i}}{\alpha_{2}}}x_{1}^{-\frac{\alpha_{1}}{\alpha_{2}}}}^{\varepsilon}\ldots\int_{\varepsilon^{-\sum_{i=N-1}^{N}\frac{\alpha_{i}}{\alpha_{N-2}}}x_{1}^{-\frac{\alpha_{1}}{\alpha_{N-2}}}\ldots x_{N-3}^{-\frac{\alpha_{N-3}}{\alpha_{N-2}}}}^{\varepsilon}\\
 & \left.\int_{\varepsilon^{-\frac{\alpha_{N}}{\alpha_{N-1}}}x_{1}^{-\frac{\alpha_{1}}{\alpha_{N-1}}}\ldots x_{N-2}^{-\frac{\alpha_{N-2}}{\alpha_{N-1}}}}^{\varepsilon}\int_{0}^{x_{1}^{-\frac{\alpha_{1}}{\alpha_{N}}}\ldots x_{N-1}^{-\frac{\alpha_{N-1}}{\alpha_{N}}}}dx_{N}dx_{N-1}dx_{N-2}\ldots dx_{2}\right\} dx_{1}
\end{aligned}
\\
= & \frac{\varepsilon^{N-\frac{1}{\alpha_{2}}\sum_{i=1}^{N}\alpha_{i}}}{\prod_{i=1,i\neq2}^{N}\left(-\frac{\alpha_{i}}{\alpha_{2}}+1\right)}+\frac{\varepsilon^{N-\frac{1}{\alpha_{3}}\sum_{i=1}^{N}\alpha_{i}}}{\prod_{i=1,i\neq3}^{N}\left(-\frac{\alpha_{i}}{\alpha_{3}}+1\right)}+\cdots+\frac{\varepsilon^{N-\frac{1}{\alpha_{N}}\sum_{i=1}^{N}\alpha_{i}}}{\prod_{i=1,i\neq N}^{N}\left(-\frac{\alpha_{i}}{\alpha_{N}}+1\right)}
\end{aligned}
\]
and
\[
1-\Sigma_{R}=\frac{\ell\left(\left.\tilde{U}_R\right\backslash \tilde{U}_{R}\left(\boldsymbol{\alpha}\right)\right)}{\ell\left(\tilde{U}_R\right)}=\sum_{j=2}^{N}\frac{\varepsilon^{-\frac{1}{\alpha_{j}}\sum_{i=1}^{N}\alpha_{i}}}{\prod_{i=1\neq j}^{N}\left(-\frac{\alpha_{i}}{\alpha_{j}}+1\right)}.
\]
Consequently, 
 $\frac{\ln\left(1-\Sigma_{R}\right)}{R}\underset{{\scriptstyle R\rightarrow-\infty}}{\rightarrow}\frac{\infty}{\infty}.$
Using L'H\^opital's rule, we obtain 
\begin{equation}\label{eq:hopital}
\frac{\frac{d}{dR}\ln\left(1-\Sigma_{R}\right)}{\frac{d}{dR}R}=\frac{\sum_{j=2}^{N}\left(-\frac{1}{\alpha_{j}}\sum_{i=1}^{N}\alpha_{i}\right)\frac{\varepsilon^{-\frac{1}{\alpha_{j}}\sum_{i=1}^{N}\alpha_{i}}}{\prod_{i=1\neq j}^{N}\left(-\frac{\alpha_{i}}{\alpha_{j}}+1\right)}}{\sum_{j=2}^{N}\frac{\varepsilon^{-\frac{1}{\alpha_{j}}\sum_{i=1}^{N}\alpha_{i}}}{\prod_{i=1\neq j}^{N}\left(-\frac{\alpha_{i}}{\alpha_{j}}+1\right)}}.
\end{equation}
Since $R<0$ is sufficiently large and therefore $\varepsilon=\textrm{e}^R>0$ is sufficiently small, the highest power of $\varepsilon$ in \eqref{eq:hopital} is the one with lowest exponent,
that is,
\[
\max_{j=2,\ldots,N}\left\{ \varepsilon^{-\frac{1}{\alpha_{j}}\sum_{i=1}^{N}\alpha_{i}}\right\} =\varepsilon^{\min_{j=2,\ldots,N}\left\{ -\frac{1}{\alpha_{j}}\sum_{i=1}^{N}\alpha_{i}\right\} }=\varepsilon^{-\frac{1}{\alpha_{\min}}\sum_{i=1}^{N}\alpha_{i}}.
\]
For $R\rightarrow-\infty$, i.e. $\varepsilon=\textrm{e}^R\rightarrow0$, \eqref{eq:hopital} becomes asymptotically
close to 
$$
-\frac{1}{\alpha_{\min}}\sum_{i=1}^{N}\alpha_{i},
$$
which yields
$$
F^{+}\left(\boldsymbol{\alpha}\right)=-\frac{1}{\alpha_{\min}}\sum_{i=1}^{N}\alpha_{i}>0.
$$

\paragraph{Case 3:} 

If $\sum_{i=1}^{N}\alpha_{i}<0$, then $\alpha_{i}\leq0$ for all $i=1,\ldots,N$ or there exist $K\geq1$ and $i_{1},\ldots,i_{K}\in\left\{ 1,\ldots,N\right\} $
such that either $\alpha_{i_{1}},\ldots,\alpha_{i_{K}}\leq0$ and $\alpha_{i}<0$ for any $i\neq i_{\text{1}},\ldots,i_{K}$ or $\alpha_{i_{1}},\ldots,\alpha_{i_{K}}>0$ and $\alpha_{i}\leq0$ for any $i\neq i_{\text{1}},\ldots,i_{K}$. From the two previous cases and the fact that $F^{+}\left(-\boldsymbol{\alpha}\right)=F^{-}\left(\boldsymbol{\alpha}\right)$, we have immediately $F^{+}\left(\boldsymbol{\alpha}\right)=0$ and $F^{-}\left(\boldsymbol{\alpha}\right)>0$.

Observe that if $\sum_{i=1}^{N}\alpha_{i}=0$ then $F^{+}\left(\boldsymbol{\alpha}\right)=F^{-}\left(\boldsymbol{\alpha}\right)=0$ in Cases 2 and 3.

\subsection{The function $F^+$ when $N=3$}\label{app:F_index3}
Note that $F^{+}\left(\boldsymbol{\alpha}\right)$ is invariant under
permutations of $\alpha_{i}$, $i=1,\ldots,N$. In particular, for
$N=3$, it follows that

\[
F^{+}\left(\alpha_{1},\alpha_{2},\alpha_{3}\right)=\begin{cases}
+\infty, & \textrm{if }\min\left\{ \alpha_{1},\alpha_{2},\alpha_{3}\right\} \geq0\\
0, & \textrm{if }\alpha_{1}+\alpha_{2}+\alpha_{3}\leq 0\\
-\dfrac{\alpha_{1}+\alpha_{2}+\alpha_{3}}{\min\left\{ \alpha_{1},\alpha_{2},\alpha_{3}\right\} }, & \textrm{if }\min\left\{ \alpha_{1},\alpha_{2},\alpha_{3}\right\} <0\textrm{ and }\alpha_{1}+\alpha_{2}+\alpha_{3}\geq 0
\end{cases}
\]
and 
\[
F^{-}\left(\alpha_{1},\alpha_{2},\alpha_{3}\right)=\begin{cases}
+\infty, & \textrm{if }\max\left\{ \alpha_{1},\alpha_{2},\alpha_{3}\right\} \leq0\\
0, & \textrm{if }\alpha_{1}+\alpha_{2}+\alpha_{3}\geq 0\\
-\dfrac{\alpha_{1}+\alpha_{2}+\alpha_{3}}{\max\left\{ \alpha_{1},\alpha_{2},\alpha_{3}\right\} }, & \textrm{if }\max\left\{ \alpha_{1},\alpha_{2},\alpha_{3}\right\} >0\textrm{ and }\alpha_{1}+\alpha_{2}+\alpha_{3}\leq 0
\end{cases}
\]
such that
\[
F^{\textrm{index}}\left(\alpha_{1},\alpha_{2},\alpha_{3}\right)=\begin{cases}
+\infty, & \textrm{if }\min\left\{ \alpha_{1},\alpha_{2},\alpha_{3}\right\} \geq0\\
-\infty, & \textrm{if }\max\left\{ \alpha_{1},\alpha_{2},\alpha_{3}\right\} \leq0\\
0, & \textrm{if }\alpha_{1}+\alpha_{2}+\alpha_{3}=0\\
\dfrac{\alpha_{1}+\alpha_{2}+\alpha_{3}}{\max\left\{ \alpha_{1},\alpha_{2},\alpha_{3}\right\} }, & \textrm{if }\max\left\{ \alpha_{1},\alpha_{2},\alpha_{3}\right\} >0\textrm{ and }\alpha_{1}+\alpha_{2}+\alpha_{3}<0\\
-\dfrac{\alpha_{1}+\alpha_{2}+\alpha_{3}}{\min\left\{ \alpha_{1},\alpha_{2},\alpha_{3}\right\} }, & \textrm{if }\min\left\{ \alpha_{1},\alpha_{2},\alpha_{3}\right\} <0\textrm{ and }\alpha_{1}+\alpha_{2}+\alpha_{3}>0.
\end{cases}
\]
\end{document}